\documentclass[reqno,10pt]{amsart}
\makeatletter
 \usepackage{enumitem}
\newcommand{\Rmnum}[1]{\expandafter\@slowromancap\romannumeral#1@}

\newtheorem{lemma}{Lemma}[section]

\newtheorem{theorem}{Theorem}[section]
\newtheorem{corollary}{Corollary}[section]
\newtheorem{remark}{Remark}[section]

\usepackage{mathrsfs}
\usepackage{color}
\usepackage{cite}
\numberwithin{equation}{section}

\title[Compressible Euler Equations]
{The time asymptotic expansion for the compressible Euler equations  with time-dependent damping}%
\author[S. Geng,  F. Huang, G. Jin  and X. Wu]{}
\email{sfgeng@xtu.edu.cn}

\email{fhuang@amt.ac.cn}
\email{jinguanghui@ybu.edu.cn }
\email{wuxc19@amss.ac.cn}

\subjclass[2000]{35L60, 35L65, 76R50, 76S05.}
 \keywords{ Compressible Euler equations,  time-dependent damping,  time asymptotic expansion}
 \thanks{* Corresponding author.}

\begin{document}

\maketitle \centerline{\scshape Shifeng Geng$^{a}$ \ \   Feimin Huang$^{b}$ \ \  Guanghui Jin$^{c}$ \ \   Xiaochun Wu$^{*, d}$}
\medskip
{\footnotesize
 \centerline{$^a\ \!$School of Mathematics and Computational Science, Xiangtan University}
 \centerline{Xiangtan 411105, China}
 \centerline{$^b\ \!$Academy of Mathematics and Systems Science, Chinese Academy of Sciences}
   \centerline{Beijing 100190, China}
  \centerline{$^c\ \!$Department of Mathematics, Yanbian University, Yanji 133002, China }
    \centerline{$^d\ \!$School of Mathematics and Statistics, Central South University }
    \centerline{Changsha 410083, China}
}

\medskip
\bigskip
\begin{abstract}
In this paper, we study the compressible
Euler equations with time-dependent damping $-\frac{1}{(1+t)^{\lambda}}\rho u$.  We  propose a time asymptotic expansion around the self-similar solution of the generalized porous media equation (GPME) and rigorously justify this  expansion as $\lambda \in (\frac17,1)$. In other word, instead of  the self-similar solution of GPME, the expansion is  the best asymptotic profile of the solution to the compressible
Euler equations with time-dependent damping.
\end{abstract}

\section{Introduction}
In this paper, we consider the compressible
Euler equations with time-dependent damping as follows:
\begin{align}\label{bE:2.1}
\begin{cases}
 \rho_t+  m_{x}=0,\\
m_{t}+\Big(\frac{ m^2}{ \rho}+p( \rho)\Big)_{x}=-\frac{1}{(1+t)^{\lambda}}  m,
\end{cases}
\end{align}
with the initial data
\begin{align}\label{bbE:2.1}
( \rho,  m)(x,0)=( \rho_0, m_0)(x)\rightarrow(\rho_{\pm}, m_{\pm})\hspace{0.2cm}\mbox{as}\quad x \rightarrow \pm \infty,
\end{align}
where $\rho_{\pm}>0$ and $m_{\pm}$ are constants.
Here $\rho=\rho(x,t), m=m(x,t)$  and $p=p(\rho)$ denote the density,  momentum
and pressure, respectively. We assume that the pressure $p(\rho)$ is
a smooth function and satisfies $p^\prime(\rho)>0$ for $\rho>0$.
The damping term $-\frac{1}{(1+t)^{\lambda}}m$ represents the time-dependent friction effect, where $0<\lambda <1$ is constant.

When $\lambda=0$, the system \eqref{bE:2.1} becomes the compressible system of Euler equations with damping modeling the compressible flow through porous media. There has a huge literature on the investigations of global existence and large time behaviors of smooth solutions to the compressible Euler equations with damping.
Among them, Hsiao and Liu [13] firstly showed that the solution of \eqref{bE:2.1} tends time-asymptotically to the self-similar solution of porous media equation (PME), called by diffusion wave. Since then,
this problem has attracted considerable attentions, see \cite{Duyn-Peletier-1977,Hsiao-Liu-1,Hsiao-Liu-2, Luo-Zeng, MingMei,Mei-2,Nishihara,Nishihara-Wang-Yang,Sideris-Thomases-Wang, Wang-Yang-2007, Zhao, Zheng, Zhu-03} and the references therein.
When $0<\lambda<1$,   Cui-Yin-Zhang-Zhu \cite{Cui-Yin-Zhang-Zhu} showed that the asymptotic behavior of the solution to the problem  \eqref{bE:2.1}  is the so-called diffusion waves in the self-similar form of
 $({\bar \rho},{\bar m})(x,t)=({\bar \rho},{\bar m})(x/\sqrt{(1+t)^{1+\lambda}})$ satisfying
 \begin{align}\label{E:1.002}
\begin{cases}
{\bar \rho}_{t}+{\bar m}_{x}=0,\\
\displaystyle{
p({\bar \rho})_{x}=-\frac{1}{(1+t)^\lambda}{\bar m},
}
\end{cases}
\ \mbox{ equivalently, } \
\displaystyle{
 \frac{1}{(1+t)^\lambda}{\bar \rho}_{t} - p({\bar \rho})_{xx}=0 \quad\mbox{(GPME)}
}
\end{align}
with
 \begin{align}\label{E:1.002-1}
 \lim_{x\rightarrow\pm\infty}( {\bar \rho}, {\bar m})(x,t)=(\rho_{\pm},0).
 \end{align}
The convergence rate obtained in \cite{Cui-Yin-Zhang-Zhu} is in the form of
 \begin{align}\label{E:1.003}
\|( \rho-{\bar \rho})(t)\|_{L^\infty}\leq
\begin{cases}
 C(1+t)^{-\frac{3}{4}(1+\lambda)},\hspace{2.0cm}0<\lambda<\frac{1}{7},\\
  C_{\varepsilon}(1+t)^{-\frac{6}{7}+\varepsilon},\hspace{2.2cm}\lambda=\frac{1}{7},\\
   C(1+t)^{\lambda-1},\hspace{2.7cm}\frac{1}{7}<\lambda<1,
\end{cases}
\end{align}
for any small $\varepsilon>0.$ Similar results were obtained in \cite{Li-Li-Mei-Zhang} for the bipolar Euler-Poisson equation with time-dependent damping. For the other interesting works on the compressible
Euler equations with time-dependent damping \eqref{bE:2.1}, see  \cite{Chen-Li-Li-Mei-Zhang,Geng-Huang-Wu,Geng-Lin-Mei,Huang-Marcati-Pan, Huang-Pan-Wang,Pan-1,Pan-2,Sugiyama-1, Wirth-1,Wirth-2,Wirth-3} and reference therein.\\

It is noted that the decay rate \eqref{E:1.003} for $\frac{1}{7}<\lambda<1$ is different from the one for $0<\lambda\leq\frac{1}{7}$. We guess that the solution $\bar {\rho}$ of GPME may not be the best time-asymptotic profile of the solution $\rho$ of \eqref{bE:2.1} for  $\frac{1}{7}<\lambda<1$. We would also like to know more information on the large time behavior of $\rho(x,t)$.
Thus, we propose a time asymptotic expansion as follows:
\begin{align}\label{abE:2.2}
\begin{cases}
& \rho=\bar \rho+\sum\limits_{i=1}^{k}(1+t)^{-i\sigma}\rho_i(\xi)+P_k=:\tilde \rho_k+P_k,\\
& m=\bar m+\sum\limits_{i=1}^{k} (1+t)^{-(i+\frac{1}{2})\sigma}m_i(\xi)+Q_k=:\tilde m_k+Q_k, \quad \xi=\frac{x}{(1+t)^{\frac{1+\lambda}{2}}}, \,\sigma=1-\lambda,
\end{cases}
\end{align}
 where $(\bar \rho,\bar m)$ is the diffusion wave of GPME given in \eqref{E:1.002}, and $(\rho_i, m_i)(\xi)$ is a solution of a       linear equation  given in section \ref{section2} below. Once \eqref{abE:2.2} is justified, the optimal decay rate, the main part and even more subsequent order terms of $( \rho-{\bar \rho})(x,t)$ are clearly known.

 For convenience, we focus on the case of $m_+=m_-=0$. And the precise statement of  our main results are as follows.
\begin{theorem}\label{theorem1.1}
For any $\lambda \in (\frac17,1)$, there exists a unique positive integer $\displaystyle k_0(\lambda)=\sup_{0<\varepsilon<\frac12}[k(\lambda)-\varepsilon],$
 where $k(\lambda)=\frac{3(1+\lambda)}{4(1-\lambda)}$ and $[\cdot]$ stands for the floor function. If the initial data $(y_0,y_1)(x)\in H^{3}(\mathbb{R})\times H^{2}(\mathbb{R})$,
where
\begin{equation*}
y_0(x)=-\int_{-\infty}^{x}( \rho_0(r)-\tilde \rho_{k_0(\lambda)}(r,0)) dr\quad \mbox{and}\quad y_1(x)= m_0(x)-\tilde m_{k_0(\lambda)}(x,0),
\end{equation*}
then there is a small constant $\delta_0>0$ such that if $\delta_1=:\|y_0\|_{3}+\|y_1\|_{2}+|\rho_+-\rho_-| \leq\delta_0$, there exists a unique and global solution $( \rho,  m)(x,t)$ of the problem \eqref{bE:2.1}-\eqref{bbE:2.1}. Moreover, the remainders $P_{k_0(\lambda)}$ and $Q_{k_0(\lambda)}$ in \eqref{abE:2.2} satisfy
\begin{align}\label{bE:2.88}
&\sum_{s=0}^{2}\Big[(1+t)^{(s+1)(1+\lambda)}\|\partial_{x}^{s}P_{k_0(\lambda)}(t)\|^2+(1+t)^{2+s(1+\lambda)}\|\partial_x^sQ_{k_0(\lambda)}(t)\|^2\Big]
\notag\\
\leq&
\begin{cases}
C\delta_1^2,\quad\mbox{if} \quad k(\lambda) \notin \mathbb{N}^{+},\\
C\delta_1^2\ln^2 (1+t),\quad\mbox{if} \quad k(\lambda) \in \mathbb{N}^{+}.\nonumber
\end{cases}
\end{align}
\end{theorem}

Using the Sobolev inequality, we can derive the following estimates.
\begin{corollary}\label{corollary1.1}
Under the assumptions of Theorem \ref{theorem1.1}, the remainders $P_{k_0(\lambda)}$ and $Q_{k_0(\lambda)}$ in \eqref{abE:2.2} satisfy
\begin{align*}
&\|\partial_x^i( P_{k_0(\lambda)},  Q_{k_0(\lambda)})(t)\|_{L^{\infty}}\\
&\leq
\begin{cases}
C\delta_1((1+t)^{-\frac{3+2i}{4}(1+\lambda)},(1+t)^{-1-\frac{1+2i}{4}(1+\lambda)}),\,\,i=0,1, \quad\mbox{if} \quad k(\lambda) \notin \mathbb{N}^{+},\\
C\delta_1((1+t)^{-\frac{3+2i}{4}(1+\lambda)}\ln (1+t),(1+t)^{-1-\frac{1+2i}{4}(1+\lambda)}\ln (1+t)),\,\,i=0,1,\quad\mbox{if} \quad k(\lambda) \in \mathbb{N}^{+}.
\end{cases}
\end{align*}
\end{corollary}

\begin{remark}
Note that $\frac{3(1+\lambda)}{4}=k(\lambda) \sigma$ and
\begin{align}
k_0(\lambda)=
\begin{cases}
[k(\lambda)], \quad \quad\mbox{if}\quad k(\lambda) \notin \mathbb{N}^{+},\\
k(\lambda)-1, \quad\,\, \mbox{if}\quad k(\lambda) \in \mathbb{N}^{+},\nonumber
\end{cases}
\end{align}
the decay rate of the remainder $P_{k_0(\lambda)}$ obtained in Corollary \ref{corollary1.1} is faster than $(1+t)^{-k_0(\lambda)\sigma}$, which justifies the expansion \eqref{abE:2.2} until the order of $k_0(\lambda)$. Nevertheless, we conjecture that the expansion
 \eqref{abE:2.2} still holds for any order.
\end{remark}


\begin{remark}
If
\[\int_{-\infty}^{\infty}( \rho_0(r)-\tilde \rho_{k_0(\lambda)}(r,0)) dr\ne 0, \quad \mbox{i.e.,}\quad y_0(x)\notin H^{3}(\mathbb{R}),\]
 we can still prove that the above results hold by replacing $\xi=\frac{x}{(1+t)^{\frac{1+\lambda}{2}}}$ with $\xi=\frac{x+x_0}{(1+t)^{\frac{1+\lambda}{2}}}$ for some shift $x_0$, where $x_0$ is determined by the initial data,
\begin{equation*}
x_0=\frac{1}{\rho_+-\rho_-}\int_{\mathbb{R}}( \rho_0(x)-\tilde \rho_{k_0(\lambda)}(x,0))dx.
\end{equation*}
\end{remark}

\begin{remark}
For the case that $m_-\neq 0$ or $m_+\neq 0$, we claim the above results still hold by introducing a correction function $(\hat \rho, \hat m)(x, t)$ with exponential decay rate to delete the gap at $x=\pm \infty$, see \cite{Cui-Yin-Zhang-Zhu,Hsiao-Liu-1} for the details.
\end{remark}

The arrangement of the present paper is as follows. In section 2, we propose the time asymptotic expansion $(\tilde \rho_k, \tilde m_k)$ of the solution to \eqref{bE:2.1}.  In section 3, we justify the expansion $(\tilde \rho_{k_0}, \tilde m_{k_0})$ and prove Theorem \ref{theorem1.1}.

\

\noindent{\bf Notations}. \ \ Throughout this paper, the symbol $C$
 will be used to represent a generic constant which is
independent of $x$ and $t$ and may vary from line to line.
$L^2(\mathbb{R})$  is the space of square integrable
real valued function defined on $\mathbb{R}$ with the norm
$ \|\cdot\|$, and $H^k(\mathbb{R})$ ($H^k$ without any
ambiguity) denotes the usual Sobolev space with the norm
$\|\cdot\|_{k}$, especially
$ \|\cdot\|_0$=$ \|\cdot\|$. In addition, for $r, s \in \mathbb N$, we adopt the convention that
\begin{equation}\label{b1}
\sum_{i=r}^sa_i=0 \quad \mbox{if} \quad s<r.\nonumber
\end{equation}

\section{The time asymptotic expansion}\label{section2}
We first list some properties on the diffusion wave $({\bar \rho}, {\bar m})(\xi)$ of GPME \eqref{E:1.002}-\eqref{E:1.002-1} as follows.
 \begin{lemma}[\cite{Cui-Yin-Zhang-Zhu, Li-Li-Mei-Zhang}]\label{lamm-01}
For the diffusion wave $({\bar \rho}, {\bar m})(\xi)$  of  \eqref{E:1.002}-\eqref{E:1.002-1}, it holds that
 \begin{align*}
&|{\bar \rho}(\xi)-\rho_+|_{\xi>0}+|\bar \rho(\xi)-\rho_-|_{\xi<0}\leq  C|\rho_+-\rho_-| e^{-c\xi^2},\notag \\
&|\partial_{x}^k\partial_{t}^l\bar \rho|\leq C|\rho_+-\rho_-| (1+t)^{-\frac{k(1+\lambda)+2l}{2}}e^{-c\xi^2},\hspace{0.2cm}k+l\geq 1, k,l\geq 0,
\notag \\
&\|\partial_{x}^k\partial_{t}^l\bar \rho\|^2\leq C|\rho_+-\rho_-| (1+t)^{-\frac{(2k-1)(1+\lambda)}{2}-2l},\hspace{0.5cm}  k+l\geq 1.
\end{align*}
 \end{lemma}

Then we consider the following time asymptotic expansion:
\begin{align}\label{*abE:2.2}
\begin{cases}
& \rho=\bar \rho(\xi)+\sum\limits_{i=1}^k(1+t)^{-i\sigma}\rho_i(\xi)+P_k=:\tilde \rho_k+P_k,\\
& m=\bar m(\xi)+\sum\limits_{i=1}^k(1+t)^{-(i+\frac{1}{2})\sigma}m_i(\xi)+Q_k=:\tilde m_k+Q_k,
\end{cases}
\end{align}
where $k$ is a positive integer determined later, and
\[ \xi=\frac{x}{(1+t)^{\frac{1+\lambda}{2}}}, \,\,\sigma=1-\lambda.\]
Note that $\bar \rho_t+\bar m_x=0$, we expect
\begin{equation}\label{bE:2.3}
((1+t)^{-i\sigma}\rho_i)_{t}+  ((1+t)^{-(i+\frac12)\sigma}m_i)_{x}=0,\nonumber
\end{equation}
which implies that
 \begin{equation}\label{bE:2.5n}
m_{i \xi}=i\sigma\rho_i+\frac{1+\lambda}{2}\xi \rho_{i \xi}.\nonumber
\end{equation}
We hope $\int_{-\infty}^{\infty}\rho_i(\xi)d\xi=0$,
which leads to
 \begin{align}\label{bE:2.5new1}
m_i(\xi)=(i\sigma-(1+\lambda))G_{i}+\frac{1+\lambda}{2}(\xi G_{i})_{\xi}, \quad G_{i}(\xi):=\int_{-\infty}^{\xi}\rho_i(\eta)d\eta.
\end{align}
Plugging \eqref{*abE:2.2} and \eqref{bE:2.5new1} into $\eqref{bE:2.1}_2$, we derive a formal hierarchy of ODEs satisfied by the functions $\bar \rho$ and $\rho_i$, $i=1, \cdots, k.$ Define the source term
\begin{equation}\label{e2.7}
S(\tilde \rho_k)=\tilde m_{kt}+\Big(\frac{ \tilde m_k^2}{ \tilde \rho_k}+p( \tilde \rho_k)\Big)_{x}+\frac{1}{(1+t)^{\lambda}}  \tilde m_k.
\end{equation}

\begin{lemma}\label{lemma2.2}
For $k \in \mathbb N^+$, the source term $S(\tilde \rho_k)$ given by \eqref{e2.7} satisfies
\begin{align}\label{2.8new}
S(\tilde \rho_k)=&\sum_{i=1}^{k}(1+t)^{-\frac{1+\lambda}{2}-i\sigma}\Big[(P^{\prime}(\bar \rho)G_{i \xi})_{\xi}+\frac{1+\lambda}{2}(\xi G_{i})_{\xi}+c_{1,i}G_{i}-c_{2,i}G_{i-1}+h_{i\xi}\Big]
\notag\\
&+(1+t)^{-\frac{1+\lambda}{2}-(k+1)\sigma}R_k,
\end{align}
where
\begin{equation}
c_{1,i}=i\sigma-(1+\lambda), \quad c_{2,i}=(i\sigma-1)((i-1)\sigma-(1+\lambda)),\ i=1, 2, \cdots, k, \label{e2.10new}\nonumber
\end{equation}
$G^0, h_i$ and $R_k$ are given by \eqref{2.21new}-\eqref{2.20new} and \eqref{2.18new} below.
\end{lemma}
\begin{proof}
From $\eqref{e2.7}$, the Taylor expansion gives that
\begin{align}\label{2.9new}
P(\tilde \rho_k)
=&P(\bar \rho)+\sum_{i=1}^{k}(1+t)^{-i\sigma}\sum_{j=1}^ih_{1,i,j}+(1+t)^{-(k+1)\sigma}R_{1,k},
\end{align}
and
\begin{align}\label{2.10new}
\frac{{\tilde m_k}^2}{\tilde \rho_k}
=&(1+t)^{-\sigma}\frac{(P(\bar \rho)_{\xi})^2}{\bar \rho}+\sum_{i=2}^{k}(1+t)^{-i\sigma}\sum_{j=1}^{i-1}h_{2,i,j}
+\sum_{i=3}^{k}(1+t)^{-i\sigma}\sum_{j=2}^{i-1}h_{3,i,j}+(1+t)^{-(k+1)\sigma}R_{2,k},
\end{align}
where
\begin{align}
h_{1,i,j}&=\frac{P^{(j)}(\bar \rho)}{j!}\sum_{l_1+\cdots+l_j=i, l_1,\cdots, l_j\geq 1}\rho_{l_1}\cdots \rho_{l_j},\quad 1\leq j\leq i,\nonumber\\
h_{2,i,j}&=\frac{(-1)^{j}(P(\bar \rho)_{\xi})^2}{(\bar \rho)^{j+1}}\sum_{l_1+\cdots+l_j=i-1, l_1,\cdots, l_j\geq 1}\rho_{l_1}\cdots \rho_{l_j}
\notag\\
&\quad+\frac{2(-1)^{j}P(\bar \rho)_{\xi}}{(\bar \rho)^{j}}\sum_{l_1+\cdots+l_j=i-1,l_1,\cdots, l_j\geq 1}m_{l_1}\rho_{l_2}\cdots \rho_{l_j}, \quad 1\leq j\leq i-1, \nonumber\\
h_{3,i,j}&=\frac{(-1)^{j}}{(\bar \rho)^{j-1}}\sum_{l_1+\cdots+l_j=i-1,l_1,\cdots, l_j\geq 1}m_{l_1}m_{l_2}\rho_{l_3}\cdots \rho_{l_j},\quad 2\leq j\leq i-1,\nonumber
\end{align}
and the remainder terms $R_{1,k}, R_{2,k}$ are some functions depending on $\bar \rho$ and $\rho_l$ with $l \in \{1, \cdots, k\}$.

On the other hand, the direct computations give that
\begin{align}\label{2.11new}
&\tilde m_{kt}+\frac{1}{(1+t)^{\lambda}} \tilde m_k
\notag\\
=&-(1+t)^{-\frac{1+\lambda}{2}}P(\bar \rho)_{\xi}+(1+t)^{-\frac{1+\lambda}{2}-\sigma}\Big[-\lambda P(\bar \rho)_{\xi}+\frac{1+\lambda}{2}(\xi P(\bar \rho)_{\xi})_{\xi}+m_1\Big]
\notag\\
&+\sum_{i=2}^k(1+t)^{-\frac{1+\lambda}{2}-i\sigma}\Big[m_i-(i\sigma-1) m_{i-1}+\frac{1+\lambda}{2}(\xi m_{i-1})_{\xi}\Big]+(1+t)^{-\frac{1+\lambda}{2}-(k+1)\sigma}R_{3,k},
\end{align}
where
\begin{equation}
R_{3,k}=-(k+\frac12)\sigma m_{k}-\frac{1+\lambda}{2}\xi m_{k \xi}.\nonumber
\end{equation}
Thus, we use $\eqref{2.9new}-\eqref{2.10new}$ and $\eqref{2.11new}$ to obtain
\begin{align}\label{2.16new}
S(\tilde \rho_k)
=&(1+t)^{-\frac{1+\lambda}{2}-\sigma}\Big[m_1+\Big(P^{(1)}(\bar \rho)\rho_{1}\Big)_{\xi}-\lambda P(\bar \rho)_{\xi}+\frac{1+\lambda}{2}(\xi P(\bar \rho)_{\xi})_{\xi}+\Big(\frac{(P(\bar \rho)_{\xi})^2}{\bar \rho}\Big)_{\xi}\Big]
\notag\\
&+\sum_{i=2}^k(1+t)^{-\frac{1+\lambda}{2}-i\sigma}\Big[m_i+\Big(P^{(1)}(\bar \rho)\rho_{i}\Big)_{\xi}+\tilde h_{i\xi}-(i\sigma-1) m_{i-1}\Big]
\notag\\
&+(1+t)^{-\frac{1+\lambda}{2}-(k+1)\sigma}R_k,
\end{align}
where
\begin{align}
&\tilde h_i=\sum_{j=2}^ih_{1,i,j}+\sum_{j=1}^{i-1}h_{2,i,j}
+\sum_{j=2}^{i-1}h_{3,i,j}+\frac{1+\lambda}{2}\xi m_{i-1},\quad \mbox{for}\quad i=2, 3, \cdots, k,\nonumber\\
&R_k=(R_{1,k}+R_{2,k})_{\xi}+R_{3,k}.\label{2.18new}
\end{align}
Moreover, substituting \eqref{bE:2.5new1} into \eqref{2.16new} yields
\begin{align}
S(\tilde \rho_k)
=&(1+t)^{-\frac{1+\lambda}{2}-\sigma}\Big[\Big(P^{(1)}(\bar \rho)\rho_{1}\Big)_{\xi}+\frac{1+\lambda}{2}(\xi G_1)_{\xi}-2\lambda G_1-\lambda P(\bar \rho)_{\xi}+\Big(\frac{1+\lambda}{2}\xi P(\bar \rho)_{\xi}+\frac{(P(\bar \rho)_{\xi})^2}{\bar \rho}\Big)_{\xi}\Big]
\notag\\
&+\sum_{i=2}^k(1+t)^{-\frac{1+\lambda}{2}-i\sigma}\Big[\Big(P^{(1)}(\bar \rho)\rho_{i}\Big)_{\xi}+\frac{1+\lambda}{2}(\xi G_i)_{\xi}+c_{1,i}G_i-c_{2,i}G_{i-1}+h_{i\xi}\Big]
\notag\\
&+(1+t)^{-\frac{1+\lambda}{2}-(k+1)\sigma}R_k,\notag
\end{align}
where
\begin{align}
&c_{1,i}=i\sigma-(1+\lambda), \label{a.2.22}\\
&c_{2,i}=(i\sigma-1)((i-1)\sigma-(1+\lambda)),\notag\\
&h_i=\tilde h_i-\frac{1+\lambda}{2}(i\sigma-1)\xi G_{i-1} \label{2.21new}
\end{align}
with $i=2, \cdots, k$. Note that $c_{1,1}=-2\lambda$, we may supply
\begin{align}
G_0&=\frac{P(\bar \rho)_{\xi}}{1+\lambda},\\
h_1&=\frac{1+\lambda}{2}\xi P(\bar \rho)_{\xi}+\frac{(P(\bar \rho)_{\xi})^2}{\bar \rho},\label{2.20new}
\end{align}
so that \eqref{2.8new} holds with $i=1, \cdots, k$.
Therefore, the proof of Lemma \ref{lemma2.2} is completed.
\end{proof}

Motivated by Lemma \ref{lemma2.2}, we define the hierarchy of ODEs as
\begin{align}\label{e2.24}
(P^{\prime}(\bar \rho)G_{i \xi})_{\xi}+\frac{1+\lambda}{2}(\xi G_{i})_{\xi}+c_{1,i}G_{i}=c_{2,i}G_{i-1}-h_{i\xi}, \quad \mbox{for}\quad i=1,2,\cdots,k,
\end{align}
so that $S(\tilde \rho_{k})=O(1)(1+t)^{-\frac{1+\lambda}{2}-(k+1)\sigma}.$
We will seek for the solution $G_{i}\in \chi^{l}(\mathbb{R})$ to $\eqref{e2.24}$,
where
\begin{align}\label{a.2.29}
\chi^{l}(\mathbb{R})=\{f:\xi^s\partial_{\xi}^r f\in L^2(\mathbb{R}), \quad  \forall \,r, s \in \{0,1,\cdots ,l\}\}
\end{align}
equipped with the norm
\begin{align}
\|f\|_{\chi^{l}(\mathbb{R})}=\Big(\sum_{0\leq s, r \leq l}\int_{\mathbb{R}} (\xi^{s}\partial_{\xi}^{r}f)^2d\xi\Big)^{\frac{1}{2}}.\notag
\end{align}
Then integrating $\eqref{e2.24}$ with respect to $\xi$ over $\mathbb R$ gives that
\begin{align}\label{bbe:2.19}
\int_{\mathbb{R}} G_{i}d\xi=\frac{c_{2,i}}{c_{1,i}}\int_{\mathbb{R}} G_{i-1}d\xi\quad \mbox{for} \quad i=1, 2, \cdots, k.
\end{align}

We obtain the existence of the smooth solution $G_{i}$ of \eqref{e2.24} and \eqref{bbe:2.19} as follows.
\begin{lemma}\label{newlemma2.1}
Let $\delta=|\rho_+-\rho_-|\ll 1$, there exists a solution $G_{i} \in \chi^{m_i}(\mathbb{R})$ to \eqref{e2.24} and \eqref{bbe:2.19} for large integer  $m_i>0$. Furthermore, it holds that
\begin{align}\label{bE:2.23}
\|G_{i}\|_{\chi^{m_i}(\mathbb{R})}\leq C\delta\notag
\end{align}
and
 \begin{equation}\label{es}
 \|\xi^{s_0}\partial^{r_0}_{\xi}G_{i}\|_{L^{\infty}}\leq C\delta,\notag
 \end{equation}
 where $0\leq s_0\leq m_i$ and $0\leq r_0\leq m_i-1$.
\end{lemma}
The detailed proof is left in the  Appendix.

\


Thanks to Lemmas \ref{lemma2.2}-\ref{newlemma2.1}, we get the estimates of source term $S(\tilde \rho_{k})$ in \eqref{e2.7}.

\begin{lemma}\label{blemma2.2}
It holds that
\begin{align}\label{bE:2.23-3}
S(\tilde \rho_{k})=O(1)\delta (1+t)^{-\frac{1+\lambda}{2}-(k+1)\sigma}
\end{align}
and
\begin{align}\label{bE:2.23-4}
\|\partial_t^j\partial_x^lS(\tilde \rho_{k})\|_{L^2_x(\mathbb{R})}^2 \leq C\delta^2 (1+t)^{-2j-2(k+1)\sigma-(l+\frac{1}{2})(1+\lambda)}, \quad j,\,l\geq 0.
\end{align}
\end{lemma}
\begin{proof}
It follows from \eqref{2.8new} and Lemma \ref{newlemma2.1} that
\begin{align}
S(\tilde \rho_{k})=(1+t)^{-\frac{1+\lambda}{2}-(k+1)\sigma}R_{k},\notag
\end{align}
where $R_{k}$ is given by \eqref{2.18new}. Then, \eqref{bE:2.23-3} and \eqref{bE:2.23-4} can be obtained by direct computations. Thus, the proof of Lemma \ref{blemma2.2} is completed.
\end{proof}

\section{The estimates of the remainder terms}\label{section3}

This section is devoted to Theorem \ref{theorem1.1} by the classical energy method with the continuation argument based on the
local existence and the a priori estimates.
For any $\lambda \in (\frac17,1)$, let
\[k_0=:k_0(\lambda)=\sup_{0<\varepsilon<\frac12}[k(\lambda)-\varepsilon]=\sup_{0<\varepsilon<\frac12}\Big[\frac{3(1+\lambda)}{4(1-\lambda)}-\varepsilon\Big]\] and the time asymptotic expansion is
\begin{align}\label{*1abE:2.2}
\begin{cases}
& \rho=\bar \rho+\sum\limits_{i=1}^{k_0}(1+t)^{-i\sigma}\rho_i(\xi)+P_{k_0}=:\tilde \rho_{k_0}+P_{k_0},\\
& m=\bar m+\sum\limits_{i=1}^{k_0} (1+t)^{-(i+\frac{1}{2})\sigma}m_i(\xi)+Q_{k_0}=:\tilde m_{k_0}+Q_{k_0}.
\end{cases}
\end{align}
We shall show that the remainder $P_{k_0}$ decays faster than $(1+t)^{-k_0\sigma}$. Denote
\begin{equation}\label{bE:2.5-1}
y=-\int_{-\infty}^{x}P_{k_0}(r,t) dr,\notag
\end{equation}
then
\begin{equation*}
y_{x}=-P_{k_0}, \quad y_t= Q_{k_0}.
\end{equation*}
Thus the system \eqref{bE:2.1} can be rewritten as a quasilinear wave equation for $y$:
\begin{align}\label{bE:2.5-2}
\begin{cases}
y_{tt}-(P^{\prime}(\tilde \rho_{k_0})y_{x})_{x}+\frac{y_t}{(1+t)^{\lambda}}=g_1+g_2+S(\tilde \rho_{k_0}),\\
(y,y_t)({x},0)=(y_0, y_1)({x}),
\end{cases}
\end{align}
where
\begin{align}\label{bE:2.5-3}
&g_1=-(P( \rho)-P(\tilde \rho_{k_0})+P^{\prime}(\tilde \rho_{k_0})y_{x})_{x},\quad g_2=-\Big(\frac{ m^2}{ \rho}-\frac{\tilde m_{k_0}^2}{\tilde \rho_{k_0}}\Big)_{x}.\notag
\end{align}

Motivated by the work of \cite{Nishihara}, we seek for the solution of \eqref{bE:2.5-2} in the following solution space
\[X_T=:\{y\in C([0,T);H^{3}(\mathbb R)), y_t \in C([0,T);H^{2}(\mathbb R))\}.\]
Since the local existence of the
solution of \eqref{bE:2.5-2} can be proved by the standard iteration method, see \cite{Matsumura-1997}, the main effort in this section is to establish
the a priori estimates for the solution.

For any $T\in (0,+\infty)$, define
\begin{align*}
N(T)^2=\sup_{0\leq t\leq T}\{\|y\|^2+\sum_{i=0}^2(1+t)^{(i+1)(1+\lambda)}\|(\partial _x^iy_t)^2+(\partial _x^iy_x)^2\|\}.
\end{align*}
We assume
\begin{align}\label{ass}
N(T)\leq
\begin{cases}
\epsilon,\quad\mbox{if} \quad k(\lambda)\notin \mathbb N^+,\\
\epsilon\ln (1+T),\quad\mbox{if} \quad k(\lambda)\in \mathbb N^+,
\end{cases}
\end{align}
where $\epsilon$ is sufficiently small and will be determined later. Then it follows from Sobolev inequality $\|\partial _{x}^if\|_{L^{\infty}}\leq C\|\partial _{x}^if\|^{\frac{1}{2}}\|\partial _{x}^{i+1}f\|^{\frac{1}{2}}$ for $i=0,1$ that
\begin{align}\label{bE:2.22}
\|\partial _{x}^iy_t\|_{L^{\infty}}+\|\partial _{x}^iy_{x}\|_{L^{\infty}}\leq
\begin{cases}
C\epsilon (1+t)^{-\frac{1}{4}(2i+3)(1+\lambda)},\quad\mbox{if} \quad k(\lambda)\notin \mathbb N^+,\\
C\epsilon (1+t)^{-\frac{1}{4}(2i+3)(1+\lambda)}\ln (1+t),\quad\mbox{if} \quad k(\lambda)\in \mathbb N^+.
\end{cases}
\end{align}

We first establish the following basic energy estimate. For abbreviation, let  $(\tilde \rho, \tilde m)$ stand for $(\tilde \rho_{k_0}, \tilde m_{k_0})$ in what follows.

\begin{lemma}\label{lemma2.2n}
For any $T>0,$ assume that $y(x,t) \in X_T$ is the solution of \eqref{bE:2.5-2}. If $\epsilon$  and $\delta$ are small, then it holds that
\begin{align}\label{zbE:2.25new}
&\int_{\mathbb{R}}\Big[(1+t)^{\beta+1}(y_t^2+y_{x}^2)+(1+t)^{\beta-\lambda}y^2\Big]d x+\int_0^t\int_{\mathbb{R}}(1+\tau)^{\beta+1-\lambda}y_t^2d xd\tau
\notag\\
&\quad+\int_0^t\int_{\mathbb{R}}(1+\tau)^{\beta}y_{x}^2d xd\tau\leq  C(N(0)^2+\delta^2+\delta\epsilon),
\end{align}
where $\beta<\lambda$.
\end{lemma}
\begin{proof}
Multiplying $\eqref{bE:2.5-2}$ by $(\alpha+t)^{\beta}y$ and integrating the result over $\mathbb{R}$,  we obtain
\begin{align}\label{bE:2.24}
&\frac{d}{dt}\int_{\mathbb{R}}[(\alpha+t)^{\beta}y_ty+\frac{(\alpha+t)^{\beta}}{2(1+t)^{\lambda}}y^2]d x+(\alpha+t)^{\beta}\int_{\mathbb{R}}P^{\prime}(\tilde \rho)y_{x}^2d x+\frac{(\alpha+t)^{\beta}}{(1+t)^{1+\lambda}}\frac{\lambda \alpha-\beta+(\lambda-\beta)t}{\alpha+t}\int_{\mathbb{R}}\frac{1}{2}y^2d x
\notag\\
=&(\alpha+t)^{\beta}\int_{\mathbb{R}}y_t^2d x+\beta (\alpha+t)^{\beta-1}\int_{\mathbb{R}}y_tyd x+(\alpha+t)^{\beta}\int_{\mathbb{R}}g_1yd x
\notag\\
&+(\alpha+t)^{\beta}\int_{\mathbb{R}}g_2yd x
+(\alpha+t)^{\beta}\int_{\mathbb{R}}S(\tilde \rho)yd x,
\end{align}
where $\alpha$ is a positive constant to be determined later.
From Lemmas \ref{lamm-01}, \ref{newlemma2.1}, the a priori assumption \eqref{bE:2.22} and the expansion \eqref{*1abE:2.2}, we have
\begin{align}\label{bE:2.25}
(\alpha+t)^{\beta}\int_{\mathbb{R}}g_1yd x&=(\alpha+t)^{\beta}\int_{\mathbb{R}}(P(\rho)-P(\tilde \rho)+P^{\prime}(\tilde \rho)y_{x})y_{x}d x\leq C\epsilon(\alpha+t)^{\beta}\int_{\mathbb{R}}y_{x}^2d x
\end{align}
and
\begin{align}\label{bE:2.27}
&(\alpha+t)^{\beta}\int_{\mathbb{R}}g_2yd x=-(\alpha+t)^{\beta}\int_{\mathbb{R}}(\frac{m^2}{\rho}-\frac{\tilde m^2}{\tilde \rho})_xyd x
\notag\\
\leq &C(\delta+\epsilon)\frac{(\alpha+t)^{\beta}}{(1+t)^{1+\lambda}}\int_{\mathbb{R}}y^2d x +C(\delta+\epsilon)(1+t)^{\beta+\lambda-1}\int_{\mathbb{R}}(y_x^2+y_t^2)d x.
\end{align}
 In addition, it is easy to check that
\begin{align}\label{be3.9new}
&\beta (\alpha+t)^{\beta-1}\int_{\mathbb{R}}y_tyd x
\leq  \nu_1\frac{(\alpha+t)^{\beta}}{(1+t)^{1+\lambda}}\int_{\mathbb{R}}y^2d x+C(\nu_1)(\alpha+t)^{\beta+\lambda-1}\int_{\mathbb{R}}y_t^2d x.
\end{align}
Moreover, from the a priori assumption \eqref{ass} and the estimates \eqref{bE:2.23-4}, for the case of $k(\lambda) \notin \mathbb N^+$ we have
\begin{align}\label{be3.9n}
&(\alpha+t)^{\beta}\int_{\mathbb{R}}S(\tilde \rho)yd x \leq  (\alpha+t)^{\beta}\|y\|\|S(\tilde \rho)\|
\leq
 C\delta \epsilon(1+t)^{-1-\sigma[(k_0+1)-k(\lambda)]-(\lambda-\beta)}
 \end{align}
 and for the case of $k(\lambda) \in \mathbb N^+$ we have
 \begin{align}\label{nbe3.9new}
&(\alpha+t)^{\beta}\int_{\mathbb{R}}S(\tilde \rho)yd x \leq  (\alpha+t)^{\beta}\|y\|\|S(\tilde \rho)\|
\leq
 C\delta \epsilon(1+t)^{-1-\sigma[(k_0+1)-k(\lambda)]-(\lambda-\beta)}\ln (1+t).
 \end{align}
Substituting \eqref{bE:2.25}-\eqref{nbe3.9new} into \eqref{bE:2.24}  and choosing $\nu_1$ small enough give that
\begin{align}\label{bE:2.28}
&\frac{d}{dt}\int_{\mathbb{R}}[(\alpha+t)^{\beta}y_ty+\frac{(\alpha+t)^{\beta}}{2(1+t)^{\lambda}}y^2]d x+\frac{(\alpha+t)^{\beta}}{2}\int_{\mathbb{R}}P^{\prime}(\tilde \rho)y_{x}^2d x+ \frac{(\alpha+t)^{\beta}}{(1+t)^{1+\lambda}}\int_{\mathbb{R}}\frac{\lambda-\beta}{4}y^2d x
\notag\\
\leq &C(\alpha+t)^{\beta}\int_{\mathbb{R}}y_t^2d x+C\delta \epsilon (1+t)^{-1-\sigma[(k_0+1)-k(\lambda)]-(\lambda-\beta)}(1+\ln (1+t)),
\end{align}
where we choose $\beta<\lambda$ such that
\begin{align*}
\frac{\lambda \alpha-\beta+(\lambda-\beta)t}{\alpha+t}=\lambda-\beta+\frac{\beta(\alpha-1)}{\alpha+t}\geq \lambda-\beta>0.
\end{align*}

We multiply $\eqref{bE:2.5-2}$ by $(\alpha+t)^{\beta+1}y_t$ and integrate the result over $\mathbb{R}$ to obtain
\begin{align}\label{bE:2.31}
& \frac{d}{dt}\int_{\mathbb{R}}[(\alpha+t)^{\beta+1}\frac{1}{2}y_t^2+(\alpha+t)^{\beta+1}P^{\prime}(\tilde \rho)\frac{1}{2}y_{x}^2]d x+\frac{(\alpha+t)^{\beta+1}}{(1+t)^{\lambda}}\Big[1-\frac{(1+t)^{\lambda}}{2(\alpha+t)}\Big]\int_{\mathbb{R}}y_t^2d x
\notag\\
\leq & C(\alpha+t)^{\beta}\int_{\mathbb{R}}y_{x}^2d x+(\alpha+t)^{\beta+1}\int_{\mathbb{R}}g_1y_td x+(\alpha+t)^{\beta+1}\int_{\mathbb{R}}g_2y_td x
+(\alpha+t)^{\beta+1}\int_{\mathbb{R}}S(\tilde \rho)y_td x.
\end{align}
A direct computation yields that
\begin{align}\label{bE:2.32}
(\alpha+t)^{\beta+1}\int_{\mathbb{R}}g_1y_td x
&=\frac{d}{dt}\Big\{(\alpha+t)^{\beta+1}\int_{\mathbb{R}}\Big[-\int_{\tilde \rho}^{\tilde \rho-y_{x}}P(s)ds-P(\tilde \rho)y_{x}+\frac{1}{2}P^{\prime}(\tilde \rho)y_{x}^2\Big]d x\Big\}
\notag\\
&\quad-(1+\beta)(\alpha+t)^{\beta}\int_{\mathbb{R}}\Big[-\int_{\tilde \rho}^{\tilde \rho-y_{x}}P(s)ds-P(\tilde \rho)y_{x}+\frac{1}{2}P^{\prime}(\tilde \rho)y_{x}^2\Big]d x
\notag\\
&\quad+(\alpha+t)^{\beta+1}\int_{\mathbb{R}}[P(\tilde \rho-y_{x})-P(\tilde \rho)-P^{\prime}(\tilde \rho)y_{x}-\frac{1}{2}P^{\prime\prime}(\tilde \rho)y_{x}^2]\tilde \rho_td x
\notag\\
&\leq \frac{d}{dt}\Big\{(\alpha+t)^{\beta+1}\int_{\mathbb{R}}\Big[-\int_{\tilde \rho}^{\tilde \rho-y_{x}}P(s)ds-P(\tilde \rho)y_{x}+\frac{1}{2}P^{\prime}(\tilde \rho)y_{x}^2\Big]d x\Big\}
\notag\\
&\quad+C(\epsilon+\delta) (\alpha+t)^{\beta}\int_{\mathbb{R}}y_{x}^2d x.
\end{align}

Next, we estimate the term $(\alpha+t)^{\beta+1}\int_{\mathbb{R}}g_2y_td x$. Note that
\begin{align}\label{bE:2.34}
g_2=-\Big(\frac{m^2}{\rho}-\frac{\tilde m^2}{\tilde \rho}\Big)_{x}=-\frac{m^2}{\rho^2}y_{xx}-\frac{2m}{\rho}y_{{x}t}+\Big(\frac{m^2}{\rho^2}-\frac{\tilde m^2}{\tilde \rho^2}\Big)\tilde \rho_{x}-\Big(\frac{2m}{\rho}-\frac{2\tilde m}{\tilde \rho}\Big)\tilde m_{x},
\end{align}
and
\begin{align}\label{bE:2.35}
\begin{cases}
&|m|\leq |\tilde m|+|y_t|\leq C(\delta+\epsilon)(1+t)^{\frac{\lambda-1}{2}};\\
&|\rho_t|+|m_{x}|\leq |\tilde \rho_t|+|y_{{x}t}|+|\tilde m_{x}| \leq C(\delta+\epsilon)(1+t)^{-1};\\
&|\rho_{x}|+|m_t|\leq |\tilde \rho_{x}|+|y_{xx}|+|\tilde m_t|+|y_{tt}|\leq C(\delta+\epsilon)(1+t)^{-\frac{1+\lambda}{2}},
\end{cases}
\end{align}
we have
\begin{align}\label{bE:2.35-1}
&(\alpha+t)^{\beta+1}\int_{\mathbb{R}}\frac{m^2}{\rho^2}y_{xx}y_td x
\notag\\
=&-\frac{1}{2}\frac{d}{dt}\Big[(\alpha+t)^{\beta+1}\int_{\mathbb{R}}\frac{m^2}{\rho^2}y_{x}^2d x\Big]+\frac{\beta+1}{2}(\alpha+t)^{\beta}\int_{\mathbb{R}}\frac{m^2}{\rho^2}y_{x}^2d x+\frac{1}{2}(\alpha+t)^{\beta+1}\int_{\mathbb{R}}\Big(\frac{m^2}{\rho^2}\Big)_ty_{x}^2d x
\notag\\
&-(\alpha+t)^{\beta+1}\int_{\mathbb{R}}\Big(\frac{m^2}{\rho^2}\Big)_{x}y_{x}y_td x
\notag\\
\geq &-\frac{1}{2}\frac{d}{dt}\Big[(\alpha+t)^{\beta+1}\int_{\mathbb{R}}\frac{m^2}{\rho^2}y_{x}^2d x\Big]-C(\delta+\epsilon)(\alpha+t)^{\beta}\int_{\mathbb{R}}y_{x}^2d x-C(\delta+\epsilon)(\alpha+t)^{\beta}\int_{\mathbb{R}}y_{t}^2d x
\end{align}
and
\begin{align}\label{bE:2.35-2}
&(\alpha+t)^{\beta+1}\int_{\mathbb{R}}\Big[\frac{2m}{\rho}y_{{x}t}-\Big(\frac{m^2}{\rho^2}-\frac{\tilde m^2}{\tilde \rho^2}\Big)\tilde \rho_{x}+\Big(\frac{2m}{\rho}-\frac{2\tilde m}{\tilde \rho}\Big)\tilde m_{x}\Big]y_td x
\notag\\
=& -(\alpha+t)^{\beta+1}\int_{\mathbb{R}}\Big(\frac{m}{\rho}\Big)_{x}y_t^2d x+(\alpha+t)^{\beta+1}\int_{\mathbb{R}}\Big[-\Big(\frac{m^2}{\rho^2}-\frac{\tilde m^2}{\tilde \rho^2}\Big)\tilde \rho_{x}+\Big(\frac{2m}{\rho}-\frac{2\tilde m}{\tilde \rho}\Big)\tilde m_{x}\Big]y_td x
\notag\\
\leq & C(\delta+\epsilon)(\alpha+t)^{\beta}\int_{\mathbb{R}}y_t^2d x+C(\delta+\epsilon)^2(\alpha+t)^{\beta}\int_{\mathbb{R}}y_{x}^2d x.
\end{align}
Thus, it follows from \eqref{bE:2.34} and \eqref{bE:2.35-1}-\eqref{bE:2.35-2} that
\begin{align}\label{3.16w}
(\alpha+t)^{\beta+1}\int_{\mathbb{R}}g_2y_td x\leq& \frac{d}{dt}\Big[(\alpha+t)^{\beta+1}\int_{\mathbb{R}}\frac{m^2}{\rho^2}\frac{1}{2}y_{x}^2d x\Big]+C(\delta+\epsilon)(\alpha+t)^{\beta}\int_{\mathbb{R}}y_{x}^2d x
\notag\\
&+C(\delta+\epsilon)(\alpha+t)^{\beta}\int_{\mathbb{R}}y_{t}^2d x.
\end{align}
In addition, it is straightforward to check from \eqref{bE:2.23-4} that
\begin{align}\label{bE:2.38}
(\alpha+t)^{\beta+1}\int_{\mathbb{R}}S(\tilde \rho)y_td x&\leq \nu_1\frac{(\alpha+t)^{\beta+1}}{(1+t)^{\lambda}}\int_{\mathbb{R}}y_t^2d x+C(\nu_1)\int_{\mathbb{R}}(\alpha+t)^{\beta+1}(1+t)^{\lambda}S^2(\tilde \rho)d x
\notag\\
&\leq \nu_1\frac{(\alpha+t)^{\beta+1}}{(1+t)^{\lambda}}\int_{\mathbb{R}}y_t^2d x+C(\nu_1)\delta^2(1+t)^{-1-2\sigma[(k_0+1)-k(\lambda)]-(\lambda-\beta)}.
\end{align}
Substituting \eqref{bE:2.32} and \eqref{3.16w}-\eqref{bE:2.38} into \eqref{bE:2.31} and choosing $\nu_1$ small enough, together with the fact that $|\xi \tilde \rho^{\prime}| \leq C \delta$,  give that
\begin{align}\label{bE:2.37}
&\frac{1}{2}\frac{d}{dt}\Big[(\alpha+t)^{\beta+1} \int_{\mathbb{R}}\Big[y_t^2+P^{\prime}(\tilde \rho)y_{x}^2+\frac{m^2}{\rho^2}y_{x}^2\Big]d x\Big]
+\frac{1}{2}\frac{(\alpha+t)^{\beta+1}}{(1+t)^{\lambda}}\int_{\mathbb{R}}y_t^2d x\notag\\
\leq &\frac{d}{dt}\Big[(\alpha+t)^{\beta+1} \int_{\mathbb{R}}\Big(\int_{\tilde \rho}^{\tilde \rho-y_{x}}P(s)ds-P(\tilde \rho)y_{x}+\frac{1}{2}P^{\prime}(\tilde \rho)y_{x}^2\Big)d x\Big]
+C(\alpha+t)^{\beta}\int_{\mathbb{R}}y_{x}^2d x\notag\\
&+C(\nu_1)\delta^2(1+t)^{-1-2\sigma[(k_0+1)-k(\lambda)]-(\lambda-\beta)}.
\end{align}

Integrating $C_1\times \eqref{bE:2.28}+\eqref{bE:2.37}$ in $(0,t)$ for large constant $C_1$ and choosing $\alpha$ large enough, we obtain
\begin{align}
&(1+t)^{\beta+1}\int_{\mathbb{R}}(y_t^2+y_{x}^2)dx+ \frac{(\alpha+t)^{\beta}}{(1+t)^{\lambda}}\int_{\mathbb{R}}y^2d x+ \int_0^t\int_{\mathbb{R}}\frac{(\alpha+\tau)^{\beta}}{(1+\tau)^{1+\lambda}}y^2d xd\tau+\int_0^t\int_{\mathbb{R}}(1+\tau)^{\beta}y_{x}^2d xd\tau
\notag\\
&\quad+\int_0^t\int_{\mathbb{R}}\frac{(\alpha+\tau)^{\beta+1}}{(1+\tau)^{\lambda}}y_t^2d xd\tau\leq C(N(0)^2+\delta^2+\delta\epsilon),\notag
\end{align}
where we have used the facts that $\beta<\lambda$ and that
\begin{align}\label{eq3.20}
\begin{cases}
(k_0+1)-k(\lambda)>0 \quad \mbox{if}\quad k(\lambda)\notin \mathbb N^+,\\
(k_0+1)-k(\lambda)=0 \quad \mbox{if}\quad k(\lambda)\in \mathbb N^+,
\end{cases}
\end{align}
from $\displaystyle k_0=\sup_{0<\varepsilon<\frac12}[k(\lambda)-\varepsilon]$ with $k(\lambda)=\frac{3(1+\lambda)}{4(1-\lambda)}$.
Thus, the proof of Lemma \ref{lemma2.2n} is completed.
\end{proof}
\begin{lemma}\label{newlemma2.2}
Assume that $y(x,t) \in X_T$ is the solution of \eqref{bE:2.5-2}. If $\epsilon$  and $\delta$ are small, it holds that
\begin{align}\label{zbE:2.25}
&\int_{\mathbb{R}}(1+t)^{\lambda+1}(y_t^2+y_{x}^2)dx+\int_0^t\int_{\mathbb{R}}(1+\tau)^{\lambda}y_{x}^2d xd\tau+\int_0^t\int_{\mathbb{R}}(1+\tau)y_t^2d xd\tau
\notag\\
\leq &
\begin{cases}
C(N(0)^2+\delta^2+\delta\epsilon), \quad\mbox{if} \quad k(\lambda)\notin \mathbb N^+,\\
C(N(0)^2+\delta^2+\delta\epsilon)\ln^2 (1+t),\quad\mbox{if} \quad k(\lambda)\in \mathbb N^+.
\end{cases}
\end{align}
\end{lemma}
\begin{proof}
Multiplying $\eqref{bE:2.5-2}$ by $(\alpha+t)^{\lambda}y$ and integrating the result over $\mathbb{R}$,  we obtain
\begin{align*}
&\frac{d}{dt}\int_{\mathbb{R}}\Big[(\alpha+t)^{\lambda}y_ty+\frac{(\alpha+t)^{\lambda}}{2(1+t)^{\lambda}}y^2\Big]d x+(\alpha+t)^{\lambda}\int_{\mathbb{R}}P^{\prime}(\tilde \rho)y_{x}^2d x+\lambda (\alpha-1)\frac{(\alpha+t)^{\lambda-1}}{(1+t)^{1+\lambda}}\int_{\mathbb{R}}\frac{1}{2}y^2d x
\notag\\
=&(\alpha+t)^{\lambda}\int_{\mathbb{R}}y_t^2d x+\lambda (\alpha+t)^{\lambda-1}\int_{\mathbb{R}}y_tyd x+(\alpha+t)^{\lambda}\int_{\mathbb{R}}g_1yd x
\notag\\
&+(\alpha+t)^{\lambda}\int_{\mathbb{R}}g_2yd x
+(\alpha+t)^{\lambda}\int_{\mathbb{R}}S(\tilde \rho)yd x.
\end{align*}
 Different from \eqref{bE:2.24}, it is difficult to use the term $\frac{(\alpha+t)^{\lambda-1}}{(1+t)^{1+\lambda}}\int_{\mathbb{R}}y^2d x$ to control $(\alpha+t)^{\lambda-1}\int_{\mathbb{R}}y_tyd x$ and $(\alpha+t)^{\lambda}\int_{\mathbb{R}}g_2yd x$. To overcome this, we choose $2\lambda-1\leq \beta<\lambda$ to get
\begin{align}
\lambda (\alpha+t)^{\lambda-1}\int_{\mathbb{R}}y_tyd x
&\leq \frac{\lambda}{2}\frac{(\alpha+t)^{\beta}}{(1+t)^{1+\lambda}}\int_{\mathbb{R}}y^2d x+\frac{\lambda}{2}(\alpha+t)^{2(\lambda-1)-\beta+1+\lambda}\int_{\mathbb{R}}y_t^2d x
\notag\\
&\leq \frac{\lambda}{2}\frac{(\alpha+t)^{\beta}}{(1+t)^{1+\lambda}}\int_{\mathbb{R}}y^2d x+\frac{\lambda}{2}(\alpha+t)^{\beta+1-\lambda}\int_{\mathbb{R}}y_t^2d x \notag
\end{align}
and
\begin{align}
&(\alpha+t)^{\lambda}\int_{\mathbb{R}}g_2yd x=-(\alpha+t)^{\lambda}\int_{\mathbb{R}}\Big(\frac{m^2}{\rho}-\frac{\tilde m^2}{\tilde \rho}\Big)_xyd x
\notag\\
\leq &C(\delta+\epsilon)\frac{(\alpha+t)^{\beta}}{(1+t)^{1+\lambda}}\int_{\mathbb{R}}y^2d x +C(\delta+\epsilon)(\alpha+t)^{2\lambda-\beta}(1+t)^{\lambda-1}\int_{\mathbb{R}}(y_x^2+y_t^2)d x
\notag\\
\leq &C(\delta+\epsilon)\frac{(\alpha+t)^{\beta}}{(1+t)^{1+\lambda}}\int_{\mathbb{R}}y^2d x +C(\delta+\epsilon)(\alpha+t)^{\lambda}\int_{\mathbb{R}}(y_x^2+y_t^2)d x.\notag
\end{align}
In addition, we use the similar method applied in \eqref{bE:2.25} and \eqref{be3.9n}-\eqref{nbe3.9new} and Lemma \ref{lemma2.2n} to get
\begin{align}\label{bE:2.24new1}
&\frac{d}{dt}\int_{\mathbb{R}}\Big[(\alpha+t)^{\lambda}y_ty+\frac{(\alpha+t)^{\lambda}}{2(1+t)^{\lambda}}y^2\Big]d x+\frac{(\alpha+t)^{\lambda}}{2}\int_{\mathbb{R}}P^{\prime}(\tilde \rho)y_{x}^2d x
\notag\\
\leq &\frac{\lambda}{2}\frac{(\alpha+t)^{\beta}}{(1+t)^{1+\lambda}}\int_{\mathbb{R}}y^2d x+C(\alpha+t)^{\beta+1-\lambda}\int_{\mathbb{R}}y_t^2d x
\notag\\
&+\begin{cases}
C\delta \epsilon (1+t)^{-1-\sigma[(k_0+1)-k(\lambda)]},  \quad  \mbox{if} \quad k(\lambda) \notin \mathbb N^+,\\
C\delta \epsilon (1+t)^{-1}\ln (1+t), \quad \mbox{if} \quad k(\lambda) \in \mathbb N^+,
\end{cases}
\end{align}
where we have used the fact \eqref{eq3.20}.

On the other hand, multiplying $\eqref{bE:2.5-2}$ by $(\alpha+t)^{1+\lambda}y_t$ and integrating the result over $\mathbb{R}$, we use the same argument in \eqref{bE:2.37} to obtain that
\begin{align}\label{bE:2.37new}
&\frac{1}{2}\frac{d}{dt}\Big[(\alpha+t)^{1+\lambda} \int_{\mathbb{R}}\Big(y_t^2+P^{\prime}(\tilde \rho)y_{x}^2+\frac{m^2}{\rho^2}y_{x}^2\Big)d x\Big]
+\frac{1}{2}\frac{(\alpha+t)^{\lambda+1}}{(1+t)^{\lambda}}\int_{\mathbb{R}}y_t^2d x\notag\\
\leq &\frac{d}{dt}\Big[(\alpha+t)^{1+\lambda} \int_{\mathbb{R}}\Big(\int_{\tilde \rho}^{\tilde \rho-y_{x}}P(s)ds-P(\tilde \rho)y_{x}+\frac{1}{2}P^{\prime}(\tilde \rho)y_{x}^2\Big)d x\Big]
\notag\\
&+C(\alpha+t)^{\lambda}\int_{\mathbb{R}}y_{x}^2d x+C\delta^2(1+t)^{-1-2\sigma[(k_0+1)-k(\lambda)]}.
\end{align}
Thanks to \eqref{eq3.20} and the improved estimates \eqref{zbE:2.25new}, integrating $C_2 \times \eqref{bE:2.24new1}+\eqref{bE:2.37new}$ in $(0,t)$ for large constant $C_2$ and choosing $\alpha$ large enough
 imply
\eqref{zbE:2.25} directly. Thus the proof of Lemma \ref{newlemma2.2} is completed.
\end{proof}

\begin{lemma}\label{blemma2.3}
Assume that $y(x,t) \in X_T$ is the solution of \eqref{bE:2.5-2}. If $\epsilon$  and $\delta$ are small, it holds that for $s=1,2$,
\begin{align}\label{bE:2.67}
&(\alpha+t)^{(s+1)(1+\lambda)}\int_{\mathbb{R}}[(\partial_{x}^{s+1}y)^2+(\partial_{x}^sy_t)^2]d x+\int_0^t\int_{\mathbb{R}}(\alpha+\tau)^{1+s(1+\lambda)}(\partial_{x}^sy_t)^2d xd\tau
\notag\\
&\quad+\int_0^t\int_{\mathbb{R}}(\alpha+\tau)^{\lambda+s(1+\lambda)}(\partial_{x}^{s+1}y)^2d xd\tau \leq
\begin{cases}
C(N(0)^2+\delta^2+\delta\epsilon), \quad \mbox{if}\quad k(\lambda)\notin \mathbb N^+,\\
C(N(0)^2+\delta^2+\delta\epsilon)\ln^2 (1+t),\quad \mbox{if}\quad k(\lambda)\in \mathbb N^+.
\end{cases}
\end{align}
\end{lemma}
\begin{proof}
Differentiating  \eqref{bE:2.5-2} with respect to $x$, we obtain
\begin{equation}\label{bE:2.42}
(\partial_{x}y)_{tt}-\partial_{x}(P^{\prime}(\tilde \rho)y_{x})_{x}+\frac{\partial_{x}y_t}{(1+t)^{\lambda}}=\partial_{x}g_1+\partial_{x}g_2+\partial_{x}S(\tilde \rho).
\end{equation}
Multiplying \eqref{bE:2.42} by $(\alpha+t)^{2(1+\lambda)}\partial_{x}y_t$ and integrating the result with respect to $x$ over $\mathbb{R}$ give that
\begin{align}\label{bE:2.44}
&\frac{d}{dt}\Big[(\alpha+t)^{2(1+\lambda)}\int_{\mathbb{R}}\frac{1}{2}(P^{\prime}(\tilde \rho)(\partial_{x}^{2}y)^2+(\partial_{x}y_t)^2)d x\Big]+\int_{\mathbb{R}}\frac{(\alpha+t)^{2(1+\lambda)}}{(1+t)^{\lambda}}(\partial_{x}y_t)^2d x
\notag\\
\leq & 2(1+\lambda)(\alpha+t)^{1+2\lambda}\int_{\mathbb{R}}\frac{1}{2}(\partial_{x}y_t)^2d x
+(\alpha+t)^{2(1+\lambda)}\int_{\mathbb{R}}\partial_{x}g_1\partial_{x}y_td x+(\alpha+t)^{2(1+\lambda)}\int_{\mathbb{R}}\partial_{x}g_2\partial_{x}y_td x
\notag\\
&+(\alpha+t)^{2(1+\lambda)}\int_{\mathbb{R}}\partial_{x}S(\tilde \rho)\partial_{x}y_td x+I_1,
\end{align}
where
\begin{align*}
I_1=&2(1+\lambda)(\alpha+t)^{1+2\lambda}\int_{\mathbb{R}}\frac{1}{2}P^{\prime}(\tilde \rho)(\partial_{x}^2y)^2d x+(\alpha+t)^{2(1+\lambda)}\int_{\mathbb{R}}\frac{1}{2}(P^{\prime}(\tilde \rho))_t(\partial_{x}^{2}y)^2d x
\notag\\
&+(\alpha+t)^{2(1+\lambda)}\int_{\mathbb{R}}[\partial_{x}^2 y(P^{\prime}(\tilde \rho))_{x}+y_{x}\partial_{x}^2(P^{\prime}(\tilde \rho))]\partial_{x}y_td x.
\end{align*}
A direct computation yields that
\begin{align}\label{bE:2.46}
|I_1|\leq &C\Big(1+\delta\frac{\alpha+t}{1+t}\Big) (\alpha+t)^{1+2\lambda}\int_{\mathbb{R}}(\partial_{x}^{2}y)^2d x+C\delta \frac{(\alpha+t)^{2(1+\lambda)}}{(1+t)^{\lambda}}(\partial_{x}y_t)^2d x
+C\delta\frac{(\alpha+t)^{2(1+\lambda)}}{(1+t)^{2+\lambda}}\int_{\mathbb{R}}y_{x}^2d x.
\end{align}

Note that
\begin{align*}
&(\alpha+t)^{2(1+\lambda)}\int_{\mathbb{R}}\partial_{x}g_1\partial_{x}y_td x
\notag\\
=&(\alpha+t)^{2(1+\lambda)}\int_{\mathbb{R}}[(P^{\prime}(\tilde \rho)-P^{\prime}(\rho))y_{xx}+(P^{\prime}(\rho)-P^{\prime}(\tilde \rho)+P^{\prime\prime}(\tilde \rho)y_{x})\tilde \rho_{x}]\partial_{x}^{2}y_td x
\notag\\
=&\frac{d}{dt}\Big[(\alpha+t)^{2(1+\lambda)}\int_{\mathbb{R}} (P^{\prime}(\tilde \rho)-P^{\prime}(\rho))\frac{1}{2}(\partial_{x}^{2}y)^2d x\Big]
+I_{2,1}+I_{2,2},
\end{align*}
where
\begin{align*}
&I_{2,1}=-2(1+\lambda)(\alpha+t)^{1+2\lambda}\int_{\mathbb{R}} (P^{\prime}(\tilde \rho)-P^{\prime}(\rho))\frac{1}{2}(\partial_{x}^{2}y)^2d x-(\alpha+t)^{2(1+\lambda)}\int_{\mathbb{R}} (P^{\prime}(\tilde \rho)-P^{\prime}(\rho))_t\frac{1}{2}(\partial_{x}^{2}y)^2d x,
\notag\\
&I_{2,2}=(\alpha+t)^{2(1+\lambda)}\int_{\mathbb{R}}[(P^{\prime}(\rho)-P^{\prime}(\tilde \rho)+P^{\prime\prime}(\tilde \rho)y_{x})_{x}\tilde \rho_{x}+
(P^{\prime}(\rho)-P^{\prime}(\tilde \rho)+P^{\prime\prime}(\tilde \rho)y_{x})\tilde \rho_{xx}]\partial_{x}y_td x,
\end{align*}
then it follows from the a priori assumption \eqref{bE:2.22} and the estimates \eqref{bE:2.35} that
\begin{align}\label{bee3.25}
&(\alpha+t)^{2(1+\lambda)}\int_{\mathbb{R}}\partial_{x}g_1\partial_{x}y_td x
\notag\\
\leq &\frac{d}{dt}\Big[(\alpha+t)^{2(1+\lambda)}\int_{\mathbb{R}} (P^{\prime}(\tilde \rho)-P^{\prime}(\rho))\frac{1}{2}(\partial_{x}^{2}y)^2d x\Big]+
C\Big(\delta+\epsilon+\delta\frac{\alpha+t}{1+t}\Big)  (\alpha+t)^{1+2\lambda}\int_{\mathbb{R}}(\partial_{x}^{2}y)^2d x
\notag\\
&+C\delta \frac{(\alpha+t)^{2(1+\lambda)}}{(1+t)^{\lambda}}(\partial_{x}y_t)^2d x+C\delta(1+t)^{\lambda}\int_{\mathbb{R}}y_{x}^2d x.
\end{align}
In addition, it holds that
\begin{align*}
&(\alpha+t)^{2(1+\lambda)}\int_{\mathbb{R}}\partial_{x}g_2\partial_{x}y_td x
\notag\\
=&-(\alpha+t)^{2(1+\lambda)}\int_{\mathbb{R}}\partial_{x}\Big[\frac{m^2}{\rho^2}y_{xx}+\frac{2m}{\rho}y_{{x}t}-\Big(\frac{m^2}{\rho^2}-\frac{\tilde m^2}{\tilde \rho^2}\Big)\tilde \rho_{x}+\Big(\frac{2m}{\rho}-\frac{2\tilde m}{\tilde \rho}\Big)\tilde m_{x}\Big]\partial_{x}y_td x
\notag\\
=&\frac{d}{dt}\Big[ (\alpha+t)^{2(1+\lambda)}\int_{\mathbb{R}}\frac{m^2}{\rho^2}\frac{1}{2}(\partial_{x}^{2}y)^2d x\Big]+I_{3,1}+I_{3,2}+I_{3,3},
\end{align*}
where
\begin{align*}
&I_{3,1}=-2(1+\lambda)(\alpha+t)^{1+2\lambda}\int_{\mathbb{R}}\frac{m^2}{\rho^2}\frac{1}{2}(\partial_{x}^{2}y)^2d x-(\alpha+t)^{2(1+\lambda)}\int_{\mathbb{R}}\Big(\frac{m^2}{\rho^2}\Big)_t\frac{1}{2}(\partial_{x}^{2}y)^2d x,
\notag\\
&I_{3,2}=-(\alpha+t)^{2(1+\lambda)}\int_{\mathbb{R}}\Big(\frac{2m}{\rho}\Big)_{x}\frac{1}{2}(\partial_{x}y_t)^2d x,
\notag\\
&I_{3,3}=(\alpha+t)^{2(1+\lambda)}\sum_{l,r\geq 0}^{ l+r=1}\int_{\mathbb{R}}C_{1}^l\Big[\Big(\frac{m^2}{\rho^2}-\frac{\tilde m^2}{\tilde \rho^2}\Big)^{(l)}(\tilde \rho_{x})^{(r)}+\Big(\frac{2m}{\rho}-\frac{2\tilde m}{\tilde \rho}\Big)^{(l)}(\tilde m_{x})^{(r)}\Big]\partial_{x}y_td x.
\end{align*}
Similarly, a tedious computation shows that
\begin{align}\label{bE:2.48}
&(\alpha+t)^{2(1+\lambda)}\int_{\mathbb{R}}\partial_{x}g_2\partial_{x}y_td x
\notag\\
\leq& \frac{d}{dt}\Big[ (\alpha+t)^{2(1+\lambda)}\int_{\mathbb{R}}\frac{m^2}{\rho^2}\frac{1}{2}(\partial_{x}^{2}y)^2d x\Big]+C(\delta+\epsilon) (\alpha+t)^{1+2\lambda}\int_{\mathbb{R}}(\partial_{x}^{2}y)^2d x
\notag\\
&+C(\delta+\epsilon) \frac{(\alpha+t)^{2(1+\lambda)}}{(1+t)^{\lambda}}\int_{\mathbb{R}}(\partial_{x}y_t)^2d x+C\delta(1+t)^{\lambda}\int_{\mathbb{R}}(y_{x}^2+y_t^2)d x.
\end{align}
Moreover, Lemma \ref{blemma2.2} yields that
\begin{align}\label{bE:2.51}
&(\alpha+t)^{2(1+\lambda)}\int_{\mathbb{R}}\partial_{x}S(\tilde \rho)\partial_{x}y_td x
\notag\\
\leq& \nu_1 \frac{(\alpha+t)^{2(1+\lambda)a}}{(1+t)^{\lambda}}\int_{\mathbb{R}}(\partial_{x}y_t)^2d x+C(\nu_1)\delta^2 (1+t)^{-1-2\sigma[(k_0+1)-k(\lambda)]},
\end{align}
where $\nu_1$ is a small constant.
Substituting \eqref{bE:2.46}-\eqref{bE:2.51} into \eqref{bE:2.44}, and choosing $\nu_1$ small and $\alpha$ large enough give that
\begin{align}\label{bE:2.52}
&\frac{d}{dt}\Big\{(\alpha+t)^{2(1+\lambda)}\int_{\mathbb{R}}\frac{1}{2}\Big[\Big(P^{\prime}(\rho)-\frac{m^2}{\rho^2}\Big)(\partial_{x}^{2}y)^2+(\partial_{x}y_t)^2\Big]d x\Big\}+\int_{\mathbb{R}}\frac{(\alpha+t)^{2(1+\lambda)}}{2(1+t)^{\lambda}}(\partial_{x}y_t)^2d x
\notag\\
\leq& C(1+\delta\frac{\alpha+t}{1+t}) (\alpha+t)^{1+2\lambda}\int_{\mathbb{R}}(\partial_{x}^{2}y)^2d x+C(\delta+\epsilon)(1+t)^{\lambda}\int_{\mathbb{R}}(y_{x}^2+y_t^2)d x
\notag\\
&+C(\nu_1)\delta^2 (1+t)^{-1-2\sigma[(k_0+1)-k(\lambda)]}.
\end{align}

It remains to estimate the term $(\alpha+t)^{1+2\lambda}\int_{\mathbb{R}}(\partial_{x}^{2}y)^2d x$. We multiply $\eqref{bE:2.42}$ by $(\alpha+t)^{1+2\lambda}\partial_{x}y$ and integrate the result over $\mathbb R$ to obtain
\begin{align}\label{bE:2.53}
&\frac{d}{dt}\Big[\int_{\mathbb{R}}(\alpha+t)^{1+2\lambda}\partial_{x}y\partial_{x}y_t+\frac{(\alpha+t)^{1+2\lambda}}{2(1+t)^{\lambda}}(\partial_{x}y)^2)d x\Big]+(\alpha+t)^{1+2\lambda}\int_{\mathbb{R}}P^{\prime}(\tilde \rho)(\partial_{x}^{2}y)^2d x
\notag\\
\leq &
(\alpha+t)^{1+2\lambda}\int_{\mathbb{R}}(\partial_{x}g_1+\partial_{x}g_2+\partial_{x}S(\tilde \rho))\partial_{x}yd x+\sum_{i=1}^3I_{4,i},
\end{align}
where
\begin{align*}
&I_{4,1}=(1+2\lambda)(\alpha+t)^{2\lambda}\int_{\mathbb{R}}\partial_{x}y\partial_{x}y_td x+(\alpha+t)^{1+2\lambda}\int_{\mathbb{R}}(\partial_{x}y_t)^2d x,
\notag\\
&I_{4,2}=(\alpha+t)^{1+2\lambda}\int_{\mathbb{R}} (P^{\prime}(\tilde \rho))_{x}y_{x}\partial_{x}^{2}yd x,
\notag\\
&I_{4,3}=\frac{1}{2}\Big[(1+2\lambda)\frac{(\alpha+t)^{2\lambda}}{(1+t)^{\lambda}}-\lambda\frac{(\alpha+t)^{1+2\lambda}}{(1+t)^{\lambda+1}}\Big]\int_{\mathbb{R}}y_{x}^2d x.
\end{align*}
It is easy to check that
\begin{align}\label{bE:2.54}
\sum_{i=1}^3|I_{4,i}|\leq &\nu_1  (\alpha+t)^{1+2\lambda}\int_{\mathbb{R}}(\partial_{x}^{2}y)^2d x+C (\alpha+t)^{1+2\lambda}\int_{\mathbb{R}}(\partial_{x}y_t)^2d x
+C(1+t)^{\lambda}\int_{\mathbb{R}}y_{x}^2d x,
\end{align}
where $\nu_1$ is a small constant.
In addition, a direct computation shows that
\begin{align}\label{bE:2.56}
&(\alpha+t)^{1+2\lambda}\int_{\mathbb{R}}\partial_{x}g_1\partial_{x}yd x+(\alpha+t)^{1+2\lambda}\int_{\mathbb{R}}\partial_{x}g_2\partial_{x}yd x
\notag\\
\leq &C(\delta+\epsilon)(\alpha+t)^{1+2\lambda}\int_{\mathbb{R}}(\partial_{x}^{2}y)^2d x+C (\delta+\epsilon)(\alpha+t)^{1+2\lambda}\int_{\mathbb{R}}(\partial_{x}y_t)^2d x
\notag\\
&+C\delta(1+t)^{\lambda}\int_{\mathbb{R}}(y_{x}^2+y_t^2)d x
\end{align}
and
\begin{align}\label{bE:2.59}
&(\alpha+t)^{1+2\lambda}\int_{\mathbb{R}}\partial_{x}S(\tilde \rho)\partial_{x}yd x
\leq \nu_1 (\alpha+t)^{1+2\lambda}\int_{\mathbb{R}}(\partial_{x}^{2}y)^2d x+C(\nu_1)\delta^2(1+t)^{-1-2\sigma[(k_0+1)-k(\lambda)]}.
\end{align}
Substituting \eqref{bE:2.54}-\eqref{bE:2.59} into \eqref{bE:2.53} and choosing $\nu_1$ small enough give that
\begin{align}\label{bE:2.60}
&\frac{d}{dt}\Big[\int_{\mathbb{R}}(\alpha+t)^{1+2\lambda}\partial_{x}y\partial_{x}y_t+\frac{(\alpha+t)^{1+2\lambda}}{2(1+t)^{\lambda}}(\partial_{x}y)^2)d x\Big]+(\alpha+t)^{1+2\lambda}\int_{\mathbb{R}}\frac{P^{\prime}(\tilde \rho)}{2}(\partial_{x}^{2}y)^2d x
\notag\\
\leq &C (\alpha+t)^{1+2\lambda}\int_{\mathbb{R}}(\partial_{x}y_t)^2d x+C(\alpha+t)^{1+2\lambda}(1+t)^{-(1+\lambda)}\int_{\mathbb{R}}(y_{x}^2+y_t^2)d x
\notag\\
&+C(\nu_1)\delta^2 (1+t)^{-1-2\sigma[(k_0+1)-k(\lambda)]}.
\end{align}
Integrating
$\eqref{bE:2.52}+C_3\times \eqref{bE:2.60}$ in $(0,t)$ for large constants $C_3$ and $\alpha$ and using Lemma \ref{newlemma2.2}, we obtain
\begin{align}
&(\alpha+t)^{2(1+\lambda)}\int_{\mathbb{R}}[(\partial_{x}^{2}y)^2+(\partial_{x}y_t)^2]d x+\int_0^t\int_{\mathbb{R}}(\alpha+\tau)^{1+2\lambda}(\partial_{x}^{2}y)^2d xd\tau
\notag\\
&\quad+\int_0^t\int_{\mathbb{R}}\frac{(\alpha+\tau)^{2(1+\lambda)}}{(1+t)^{\lambda}}(\partial_{x}y_t)^2d xd\tau\leq
\begin{cases}
C(N(0)^2+\delta^2+\delta\epsilon), \quad\mbox{if} \quad k(\lambda)\notin \mathbb N^+,\\
C(N(0)^2+\delta^2+\delta\epsilon)\ln^2 (1+t),\quad\mbox{if} \quad k(\lambda)\in \mathbb N^+,\notag
\end{cases}
\end{align}
where we have used the fact \eqref{eq3.20}.

Similarly, we can obtain the desired estimates \eqref{bE:2.67} for the case of $s=2$.
Thus, the proof is completed.
\end{proof}

\begin{lemma}\label{blemma2.4}
Assume that $y(x,t) \in X_T$ is the solution of \eqref{bE:2.5-2}. If $\epsilon$  and $\delta$ are small, it holds that for $s=0,1$,
\begin{align}\label{bbE:2.86}
&(1+t)^{3+\lambda+s(1+\lambda)}\int_{\mathbb{R}}[(\partial_{x}^sy_{tt})^2+(\partial_{x}^sy_{{x}t})^2]d x+\int_0^t\int_{\mathbb{R}}(1+\tau)^{3+s(1+\lambda)}(\partial_{x}^sy_{tt})^2d xd\tau
\notag\\&
+\int_0^t\int_{\mathbb{R}}(1+\tau)^{2+\lambda+s(1+\lambda)}(\partial_{x}^sy_{{x}t})^2d xd\tau\leq \begin{cases}
C(N(0)^2+\delta^2+\delta\epsilon), \quad\mbox{if} \quad k(\lambda)\notin \mathbb N^+,\\
C(N(0)^2+\delta^2+\delta\epsilon)\ln^2 (1+t),\quad\mbox{if} \quad k(\lambda)\in \mathbb N^+.
\end{cases}
\end{align}
\end{lemma}
\begin{proof}
Differentiating \eqref{bE:2.5-2} with respect to $t$ gives
\begin{align}\label{bE:2.68}
y_{ttt}-(P^{\prime}(\tilde \rho)y_{x})_{xt}+\frac{y_{tt}}{(1+t)^{\lambda}}-\frac{\lambda y_t}{(1+t)^{1+\lambda}} =g_{1t}+g_{2t}+S(\tilde \rho)_{t}.
\end{align}
 Then we multiply \eqref{bE:2.68} by $(\alpha+t)^{3+\lambda}y_{tt}$ and integrate the result over $\mathbb{R}$ to obtain that
\begin{align}\label{bE:2.69}
&\frac{d}{dt}\int_{\mathbb{R}}\frac{(\alpha+t)^{3+\lambda}}{2}[y_{tt}^2+P^{\prime}(\tilde \rho)y_{tx}^2]d x+\frac{(\alpha+t)^{3+\lambda}}{(1+t)^{\lambda}}\int_{\mathbb{R}}y_{tt}^2d x
\notag\\
\leq &\frac{3+\lambda}{2}(\alpha+t)^{2+\lambda}\int_{\mathbb{R}} y_{tt}^2d x+(\alpha+t)^{3+\lambda}\int_{\mathbb{R}}g_{1t}y_{tt}d x+(\alpha+t)^{3+\lambda}\int_{\mathbb{R}}g_{2t}y_{tt}d x
\notag\\
&+(\alpha+t)^{3+\lambda}\int_{\mathbb{R}}S(\tilde \rho)_{t}y_{tt}d x+I_5,
\end{align}
where
\begin{align*}
I_5=&\frac{3+\lambda}{2}(\alpha+t)^{2+\lambda}\int_{\mathbb{R}}P^{\prime}(\tilde \rho)y_{tx}^2d x+(\alpha+t)^{3+\lambda}\Big[\int_{\mathbb{R}}P^{\prime\prime}(\tilde \rho)\tilde \rho_t \frac{1}{2}y_{tx}^2d x+\int_{\mathbb{R}}P^{\prime\prime}(\tilde \rho)\tilde \rho_t \tilde \rho_{x}y_{x}y_{tt}d x
\notag\\
&+\int_{\mathbb{R}}P^{\prime\prime}(\tilde \rho)\tilde \rho_{tx}y_{x}y_{tt}d x+\int_{\mathbb{R}}P^{\prime\prime}(\tilde \rho)\tilde \rho_t y_{xx}y_{tt}d x\Big]+\lambda\frac{(\alpha+t)^{3+\lambda}}{(1+t)^{1+\lambda}}\int_{\mathbb{R}}y_ty_{tt}d x.
\end{align*}
A direct computation shows that
\begin{align}\label{bE:2.70}
|I_5|\leq& \nu_1\frac{(\alpha+t)^{3+\lambda}}{(1+t)^{\lambda}}\int_{\mathbb{R}}y_{tt}^2d x+C(\alpha+t)^{3+\lambda}\Big[\frac{1}{\alpha+t}+\frac{\delta}{1+t}\Big] \int_{\mathbb{R}}y_{tx}^2d x+C\delta^2 \frac{(\alpha+t)^{3+\lambda}}{(1+t)^{3}}\int_{\mathbb{R}}y_{x}^2d x
\notag\\
&+C\delta^2 \frac{(\alpha+t)^{3+\lambda}}{(1+t)^{2-\lambda}}\int_{\mathbb{R}}y_{xx}^2d x+\frac{(\alpha+t)^{3+\lambda}}{(1+t)^{\lambda+2}}\int_{\mathbb{R}}y_{t}^2d x,
\end{align}
\begin{align}
(\alpha+t)^{3+\lambda}\int_{\mathbb{R}}g_{1t}y_{tt}d x
\leq& \frac{d}{dt}\Big[(\alpha+t)^{3+\lambda}\int_{\mathbb{R}}(P^{\prime}(\tilde \rho)-P^{\prime}(\rho))\frac{1}{2}y_{{x}t}^2d x\Big]+\nu_1\frac{(\alpha+t)^{3+\lambda}}{(1+t)^{\lambda}}\int_{\mathbb{R}}y_{tt}^2d x
\notag\\
&+C(\nu_1)\delta^2\frac{(\alpha+t)^{3+\lambda}}{(1+t)^{3}}\int_{\mathbb{R}}y_{x}^2d x+C(\nu_1)\delta^2\frac{(\alpha+t)^{3+\lambda}}{(1+t)^{2-\lambda}}\int_{\mathbb{R}}y_{xx}^2d x
\notag\\
&+C\epsilon(\alpha+t)^{2+\lambda}\int_{\mathbb{R}}y_{{x}t}^2d x+C(\delta+\epsilon)\frac{(\alpha+t)^{3+\lambda}}{1+t}\int_{\mathbb{R}}y_{{x}t}^2d x
\end{align}
and
\begin{align}
&(\alpha+t)^{3+\lambda}\int_{\mathbb{R}}g_{2t}y_{tt}d x
\notag\\
\leq& \frac{d}{dt}\Big[(\alpha+t)^{3+\lambda}\int_{\mathbb{R}}\frac{m^2}{\rho^2}\frac{1}{2}(y_{{x}t})^2d x\Big]
+[\nu_1+C(\delta+\epsilon)]\frac{(\alpha+t)^{3+\lambda}}{1+t}\int_{\mathbb{R}}y_{tt}^2d x
\notag\\
&+C(\nu_1)(\delta+\epsilon)^2\frac{(\alpha+t)^{3+\lambda}}{(1+t)^{2-\lambda}}\int_{\mathbb{R}}y_{xx}^2d x+C(\delta+\epsilon)^2\frac{(\alpha+t)^{3+\lambda}}{1+t}\int_{\mathbb{R}}y_{{x}t}^2d x
\notag\\
&+C(\nu_1)\delta^2(1+t)^{\lambda}\int_{\mathbb{R}}y_{x}^2d x
+C(\nu_1)\delta^2(1+t)\int_{\mathbb{R}}y_t^2d x.
\end{align}
In addition, it follows from Lemma \ref{blemma2.2} that
\begin{align}\label{bE:2.74}
(\alpha+t)^{3+\lambda}\int_{\mathbb{R}}S(\tilde \rho)_{t}y_{tt}d x\leq \nu_1\frac{(\alpha+t)^{3+\lambda}}{(1+t)^{\lambda}}\int_{\mathbb{R}}y_{tt}^2d x+C(\nu_1)\delta^2 (1+t)^{-1-2\sigma[(k_0+1)-k(\lambda)]},
\end{align}
where $\nu_1$ is a small constant. Next,
substituting \eqref{bE:2.70}-\eqref{bE:2.74} into \eqref{bE:2.69} and choosing $\alpha$ large and $\nu_1$ small enough yield that
\begin{align}\label{bE:2.75}
&\frac{d}{dt}\Big\{\int_{\mathbb{R}}\frac{(\alpha+t)^{3+\lambda}}{2}\Big[y_{tt}^2+\Big(P^{\prime}(\rho)-\frac{m^2}{\rho^2}\Big)y_{{x}t}^2\Big]d x\Big\}+\frac{(\alpha+t)^{3+\lambda}}{(1+t)^{\lambda}}\int_{\mathbb{R}}y_{tt}^2d x
\notag\\
\leq & C(\alpha+t)^{3+\lambda}\Big[\frac{1}{\alpha+t}+\frac{\delta+\epsilon}{1+t}\Big] \int_{\mathbb{R}}y_{{x}t}^2d x+C(\delta+\epsilon)^2 \frac{(\alpha+t)^{3+\lambda}}{(1+t)^{2-\lambda}}\int_{\mathbb{R}}y_{xx}^2d x+C\delta^2 (1+t)^{\lambda}\int_{\mathbb{R}}y_{x}^2d x
\notag\\
&+C(1+t)\int_{\mathbb{R}}y_{t}^2d x+C(\nu_1)\delta^2 (1+t)^{-1-2\sigma[(k_0+1)-k(\lambda)]}.
\end{align}
From  Lemmas \ref{newlemma2.2}-\ref{blemma2.3} and the fact \eqref{eq3.20}, we integrate \eqref{bE:2.75} over $[0,t]$ to obtain that
\begin{align*}
(\alpha+t)^{3+\lambda}\int_{\mathbb{R}}(y_{tt}^2+y_{{x}t}^2)d x+\int_0^t\int_{\mathbb{R}}(\alpha+\tau)^{3}y_{tt}^2d xd\tau
\leq \begin{cases}
C(N(0)^2+\delta^2+\delta\epsilon), \quad\mbox{if} \quad k(\lambda)\notin \mathbb N^+,\\
C(N(0)^2+\delta^2+\delta\epsilon)\ln^2 (1+t),\quad\mbox{if} \quad k(\lambda)\in \mathbb N^+.
\end{cases}
\end{align*}

Similarly,
we can verify that \eqref{bbE:2.86} holds for the case of $s=1$.
Thus the proof is completed.
\end{proof}
\begin{lemma}\label{blemma2.5}
Assume that $y(x,t) \in X_T$ is the solution of \eqref{bE:2.5-2}. If $\epsilon$  and $\delta$ are small, it holds that
\begin{align}\label{bbE:2.86new}
&(1+t)^2\int_{\mathbb{R}}y_t^2d x \leq \begin{cases}
C(N(0)^2+\delta^2+\delta\epsilon), \quad\mbox{if} \quad k(\lambda)\notin \mathbb N^+,\\
C(N(0)^2+\delta^2+\delta\epsilon)\ln^2 (1+t),\quad\mbox{if} \quad k(\lambda)\in \mathbb N^+.
\end{cases}
\end{align}
\end{lemma}
\begin{proof}
Multiplying $\eqref{bE:2.68}$ by $(1+t)^{2+\lambda}y_t$ and integrating the result over $\mathbb{R}$ lead to
\begin{align}\label{bE:2.76}
&\frac{d}{dt}\Big[\int_{\mathbb{R}}(1+t)^{2+\lambda}y_ty_{tt}d x+\int_{\mathbb{R}}\frac{1}{2}(1+t)^{2}y_t^2d x\Big]
\notag\\
\leq &C(1+t)^{2+\lambda}\int_{\mathbb{R}}(y_{{x}t}^2+y_{tt}^2)d x+C(1+t)^{\lambda}\int_{\mathbb{R}}y_{x}^2d x+C(1+t)\int_{\mathbb{R}}y_t^2d x
\notag\\
&
+(1+t)^{2+\lambda}\int_{\mathbb{R}}(g_{1t}+g_{2t}+S(\tilde \rho)_{t})y_{t}d x.
\end{align}
It follows from \eqref{bE:2.35} that
\begin{align}\label{bE:2.78}
&(1+t)^{2+\lambda}\int_{\mathbb{R}}g_{1t}y_{t}d x+(1+t)^{2+\lambda}\int_{\mathbb{R}}g_{2t}y_{t}d x
\notag\\
\leq &C(1+t)^{2+\lambda}\int_{\mathbb{R}}(y_{{x}t}^2+y_{tt}^2)d x
+C(1+t)\int_{\mathbb{R}}y_t^2d x+C(\nu_1)\delta^2(1+t)^{\lambda}\int_{\mathbb{R}}y_{x}^2d x.
\end{align}
Moreover, it follows from Lemma \ref{blemma2.2} that
\begin{align}\label{bE:2.81}
(1+t)^{2+\lambda}\int_{\mathbb{R}}S(\tilde \rho)_{t}y_{t}d x\leq (1+t)\int_{\mathbb{R}}y_t^2d x+C(\nu_1)\delta^2 (1+t)^{-1-2\sigma[(k_0+1)-k(\lambda)]}.
\end{align}
Substituting \eqref{bE:2.78}-\eqref{bE:2.81} into \eqref{bE:2.76} and integrating the result over $[0,t]$, we deduce \eqref{bbE:2.86new} from
 Lemmas \ref{newlemma2.2}-\ref{blemma2.3}. Thus the proof is completed.
\end{proof}

\begin{proof}[\bf Proof of Theorem \ref{theorem1.1}]
Lemmas \ref{newlemma2.2}-\ref{blemma2.5} show that there exists some positive constant $C_0$ such that
\begin{align}\label{be:2.88}
&\sum_{s=0}^{2}[(1+t)^{(s+1)(1+\lambda)}\|\partial_{x}^{s}y_{x}(t)\|^2+(1+t)^{2+s(1+\lambda)}\|\partial_x^sy_t(t)\|^2]
\notag\\
\leq &\begin{cases}
C_0(N(0)^2+\delta^2+\delta\epsilon), \quad\mbox{if} \quad k(\lambda)\notin \mathbb N^+,\\
C_0(N(0)^2+\delta^2+\delta\epsilon)\ln^2 (1+t),\quad\mbox{if} \quad k(\lambda)\in \mathbb N^+
\end{cases}
\end{align}
provided that $\epsilon \ll 1$. Choose $\epsilon=4C_0(N(0)+\delta)$ and suppose that $N(0)+\delta\ll 1$, then we can obtain from \eqref{be:2.88} that
\begin{align}
N(T)\leq
\begin{cases}
 \frac{\epsilon}{\sqrt 2},\quad\mbox{if} \quad k(\lambda)\notin \mathbb N^+,\\
 \frac{\epsilon}{\sqrt 2}\ln (1+T),\quad\mbox{if} \quad k(\lambda)\in \mathbb N^+,\notag
\end{cases}
\end{align}
 which closes the a priori assumption \eqref{ass}.
Therefore the proof is completed.
\end{proof}

\appendix
\renewcommand\appendix{\par
    \gdef\thesection{Appendix \Alph{section}}}

\section{Proof of Lemma \ref{newlemma2.1}}
For the case $i=1$,
we consider the following ODE
\begin{align}\label{bE:5.17}
(P^{\prime}(\bar \rho)G_{1 \xi})_{\xi}+\frac{1+\lambda}{2}(\xi G_{1})_{\xi}-2\lambda G_{1}=\lambda P(\bar \rho)_{\xi}-h_{1\xi}
\end{align}
 with the condition
\begin{align}\label{bE:5.15}
\int_{\mathbb{R}} G_{1}d\xi=-\frac{P(\rho_+)-P(\rho_-)}{2}.
\end{align}
Let $\tilde G_{1}=G_{1}+\frac{1}{2}P(\bar \rho)_{\xi}$, then  \eqref{bE:5.17}-\eqref{bE:5.15} can be  rewritten as
\begin{align}\label{bE:5.18}
\begin{cases}
(P^{\prime}( \rho_+)\tilde G_{1 \xi})_{\xi}+\frac{1+\lambda}{2}(\xi \tilde G_{1})_{\xi}-2\lambda \tilde G_{1}
=((P^{\prime}( \rho_+)-P^{\prime}(\bar \rho))\tilde G_{1 \xi})_{\xi}+\tilde h_{1\xi},\\
\int_{\mathbb{R}} \tilde {G}_1(\xi)d\xi=0,\notag
\end{cases}
\end{align}
where
\[\tilde h_1=\frac{1}{2}P^{\prime}(\bar \rho)P(\bar \rho)_{\xi \xi}+\frac{1+\lambda}{4}\xi P(\bar \rho)_{\xi}-h_{1}.\]
Taking the Fourier transformation of $\tilde {G}_1(\xi)$ gives that
\begin{align}\label{bE:5.3}
\begin{cases}
\frac{1+\lambda}{2}\eta \mathscr{F}_{1 \eta}+(2\lambda+P^{\prime}(\rho_+)\eta^2)\mathscr{F}_1=-\mathscr{F}[((P^{\prime}( \rho_+)-P^{\prime}(\bar \rho))\tilde G_{1 \xi})_{\xi}]-\mathscr{F}[\tilde h_{1\xi}],\\
\mathscr{F}_{1}(0)=0,
\end{cases}
\end{align}
where $\mathscr{F}_1(\eta)=\mathscr{F}[\tilde G_{1}(\xi)]$.
We construct the following iterative sequences $\{\mathscr{F}_1^n\}$:
\begin{align}\label{bE:5.4}
\begin{cases}
\frac{1+\lambda}{2}\eta \mathscr{F}^{n+1}_{1\eta}+(2\lambda+P^{\prime}(\rho_+)\eta^2)\mathscr{F}_1^{n+1}=-\mathscr{F}[((P^{\prime}( \rho_+)-P^{\prime}(\bar \rho))\tilde G^n_{1\xi})_{\xi}]-\mathscr{F}[\tilde h_{1\xi}],\\
\mathscr{F}_{1}^{n+1}(0)=0,
\end{cases}
\end{align}
where $\tilde G_1^n(\xi)=\mathscr{F}^{-1}[\mathscr{F}_1^n(\eta)]$ with $n\in \mathbb N^+$ and $\tilde G_1^0(\xi)=0$. It follows from \eqref{bE:5.4} that
\begin{align*}
\mathscr{F}_1^{n+1}(\eta)=-\frac{2}{1+\lambda}\eta^{-\frac{4\lambda}{1+\lambda}}e^{-\frac{P^{\prime}(\rho_+)}{1+\lambda}\eta^2}\int_0^{\eta}\eta_1^{\frac{4\lambda}{1+\lambda}-1}e^{\frac{P^{\prime}(\rho_+)}{1+\lambda}\eta_1^2}\Big(\mathscr{F}[((P^{\prime}( \rho_+)-P^{\prime}(\bar \rho))\tilde G^n_{1\xi})_{\xi}]+\mathscr{F}[\tilde h_{1\xi}]\Big)d\eta_1.
\end{align*}

We claim that the sequences $\{\mathscr{F}_1^{n}\}\subset \mathscr{S}(\mathbb{R})$, where $\mathscr{S}(\mathbb{R})$ is the Schwartz space. 
Indeed, since $\tilde G_1^0(\xi)=0$ and $\mathscr{F}[\tilde h_{1\xi}]\in \mathscr{S}(\mathbb{R})$, it is straightforward to imply that for any nonnegative integers $\alpha$ and $\beta$,
\begin{equation}\label{a.6}
|\eta^{\alpha}\partial_{\eta}^{\beta}\mathscr{F}_1^{1}|\rightarrow 0, \quad \mbox{as}\quad |\eta|\rightarrow \infty.\notag
\end{equation}
Thus $\mathscr{F}_1^{1} \in \mathscr{S}(\mathbb{R})$ holds. In the same way, we can verify $\{\mathscr{F}_1^{n}\}\subset \mathscr{S}(\mathbb{R})$.

It follows from the above claim that $\{\mathscr{F}_1^n\} \subset \chi^{m_1}(\mathbb{R})$ for any given integer $m_1>0$, where $ \chi^{m_1}(\mathbb{R})$ is given in \eqref{a.2.29}. We will use the contraction principle to show that $\mathscr{F}_1^n$ has a unique limit in $\chi^{m_1}(\mathbb{R})$. To this end,
let $\Delta_1^{n+1}=\mathscr{F}_1^{n+1}-\mathscr{F}_1^n$ and we hope that
\begin{align}\label{bE:5.8}
\|\Delta_1^{n+1}\|_{\chi^{m_1}(\mathbb{R})}\leq C\delta \|\Delta_1^{n}\|_{\chi^{m_1}(\mathbb{R})}\leq \frac12\|\Delta_1^{n}\|_{\chi^{m_1}(\mathbb{R})}.
\end{align}
It remains to prove \eqref{bE:5.8}. Note that $\Delta_1^{n+1}$ satisfies that
\begin{align}\label{bE:5.9}
\frac{1+\lambda}{2}\eta \Delta_{1\eta}^{n+1}+(2\lambda+P^{\prime}(\rho_+)\eta^2)\Delta_1^{n+1}=-\mathscr{F}[((P^{\prime}( \rho_+)-P^{\prime}(\bar \rho))(\tilde G^n_{1\xi}-\tilde G^{n-1}_{1\xi}))_{\xi}],
\end{align}
we take the procedure as $\eqref{bE:5.9}\times \bar \Delta_1^{n+1}+\bar{\eqref{bE:5.9}} \times \Delta_1^{n+1}$ and integrate the result over $\mathbb{R}$ to obtain
\begin{align}\label{a.9}
&\quad\int_{\mathbb{R}} P^{\prime}(\rho_+)\eta^2|\Delta_1^{n+1}|^2d\eta
\notag\\
&=\Big(\frac{1+\lambda}{4}-2\lambda\Big)\int_{\mathbb{R}}|\Delta_1^{n+1}|^2d\eta-\frac12\int_{\mathbb{R}}\mathscr{F}[((P^{\prime}( \rho_+)-P^{\prime}(\bar \rho))(\tilde G^n_{1\xi}-\tilde G^{n-1}_{1\xi}))_{\xi}]\bar \Delta_1^{n+1}d\eta
\notag\\
&\quad-\frac12\int_{\mathbb{R}}\bar {\mathscr{F}}[((P^{\prime}( \rho_+)-P^{\prime}(\bar \rho))(\tilde G^n_{1\xi}-\tilde G^{n-1}_{1\xi}))_{\xi}]\Delta_1^{n+1}d\eta
\notag\\
&\leq \tilde C\int_{\mathbb{R}}|\Delta_1^{n+1}|^2d\eta+ \nu_1 \int_{\mathbb{R}} \eta^2|\Delta_1^{n+1}|^2d\eta
+C(\nu_1)\int_{\mathbb{R}}|F[(P^{\prime}( \rho_+)-P^{\prime}(\bar \rho))(\tilde G^n_{1\xi}-\tilde G^{n-1}_{1\xi})]|^2d\eta
\notag\\
&= \tilde C\int_{\mathbb{R}}|\Delta_1^{n+1}|^2d\eta+ \nu_1 \int_{\mathbb{R}} \eta^2|\Delta_1^{n+1}|^2d\eta
+C(\nu_1)\int_{\mathbb{R}}|(P^{\prime}( \rho_+)-P^{\prime}(\bar \rho))(\tilde G^n_{1\xi}-\tilde G^{n-1}_{1\xi})|^2d\eta
\notag\\
&\leq
\tilde C\int_{-M}^M|\Delta_1^{n+1}|^2d\eta+\tilde C\int_{|\eta|\geq M}|\Delta_1^{n+1}|^2d\eta+\nu_1 \int_{\mathbb{R}} \eta^2|\Delta_1^{n+1}|^2d\eta+C(\nu_1)\delta^2\int_{\mathbb{R}}\eta^2 |\Delta_1^{n}|^2d\eta,
\end{align}
where $\bar f$ represents the conjugate complex of $f$ and we have used Plancherel's Theorem in the last two inequalities.
By choosing $\nu_1$ small and $M$ large enough so that $P^{\prime}(\rho_+)M^2>2\tilde C$, we deduce from \eqref{a.9} that
\begin{align}\label{bE:5.11}
&\int_{\mathbb{R}} \eta^2|\Delta_1^{n+1}|^2d\eta \leq C\int_{-M}^M|\Delta_1^{n+1}|^2d\eta+C(\nu_1)\delta^2\int_{\mathbb{R}}\eta^2 |\Delta_1^{n}|^2d\eta.
\end{align}
Moreover, the explicit expression $\Delta_1^{n+1}$ from \eqref{bE:5.9} shows that
\begin{align*}
&\int_{-M}^M |\Delta_1^{n+1}|^2d\eta
\notag\\
=&\int_{-M}^M\Big|\frac{2}{1+\lambda}\eta^{-\frac{4\lambda}{1+\lambda}}e^{-\frac{P^{\prime}(\rho_+)}{1+\lambda}\eta^2}\int_0^{\eta} \eta_1^{\frac{4\lambda}{1+\lambda}}e^{\frac{P^{\prime}(\rho_+)}{1+\lambda}\eta_1^2}\{\mathscr{F}[((P^{\prime}( \rho_+)-P^{\prime}(\bar \rho))(\tilde G^n_{1\xi}-\tilde G^{n-1}_{1\xi})\}d\eta_1\Big|^2d\eta
\notag\\
\leq &C\int_{-M}^M\int_{\mathbb{R}}|\mathscr{F}[((P^{\prime}( \rho_+)-P^{\prime}(\bar \rho))(\tilde G^n_{1\xi}-\tilde G^{n-1}_{1\xi})]|^2d\eta_1d\eta \leq  C_M\delta^2\int_{\mathbb{R}}\eta^2 |\Delta_1^{n}|^2d\eta,
\end{align*}
which, together with \eqref{bE:5.11}, yields that
\begin{align}\label{a.11}
&\int_{\mathbb{R}} |\Delta_1^{n+1}|^2d\eta+\int_{\mathbb{R}} \eta^2|\Delta_1^{n+1}|^2d\eta
\leq C(\nu_1,M)\delta^2\int_{\mathbb{R}}\eta^2 |\Delta_1^{n}|^2d\eta\leq\frac12\int_{\mathbb{R}}\eta^2 |\Delta_1^{n}|^2d\eta.\nonumber
\end{align}
In the same way, we can verify further that \eqref{bE:5.8} holds.

 Thus, it follows from the contraction mapping principle that
\eqref{bE:5.3} admits a unique solution $\mathscr{F}_1 \in \chi^{m_1}(\mathbb{R})$. Moreover, applying the similar argument in  \eqref{bE:5.8} we get
\begin{equation}
\|\mathscr{F}_1\|_{\chi^{m_1}(\mathbb{R})}\leq C\delta,\quad \mbox{or equivalently}, \quad \|\tilde G_{1}\|_{\chi^{m_1}(\mathbb{R})}\leq C\delta.\nonumber
\end{equation}
 Note that $\tilde G_{1}=G_{1}+\frac{1}{2}P(\bar \rho)_{\xi}$, it holds that $\|G_{1}\|_{\chi^{m_1}(\mathbb{R})}\leq C\delta.$

Similarly, for the general case $i\geq 2$, let $ \tilde G_i=G_i-\frac{c_{2,i}}{c_{1,i}}G_{i-1}$ and
we can rewrite \eqref{e2.24} and \eqref{bbe:2.19} as
\begin{align}\label{bbE:2.18}
\begin{cases}
(P^{\prime}(\rho^+)\tilde G_{i\xi})_{\xi}+\frac{1+\lambda}{2}(\xi \tilde G_{i})_{\xi}+c_{1,i}\tilde G_{i}=((P^{\prime}(\rho^+)-P^{\prime}(\bar \rho))\tilde G_{i\xi})_{\xi}+\tilde h_{i\xi},\\
\int_{\mathbb{R}} \tilde {G}_i(\xi)d\xi=0,
\end{cases}
\end{align}
where $\tilde h_i=-\frac{c_{2,i}}{c_{1,i}}\Big(P^{\prime}(\bar \rho)G_{i-1\xi}+\frac{1+\lambda}{2}\xi G_{i-1}\Big)-h_i$ and $c_{1,i}, c_{2,i},h_i$ are given in \eqref{a.2.22}-\eqref{2.21new}. Note that $G_1 \in \chi^{m_1}(\mathbb{R})$ with any given integer $m_1>0$, which implies that $\mathscr{F}[\tilde h_{i\xi}]\in \mathscr{S}(\mathbb{R})$.
Thus in the same way we can see that \eqref{bbE:2.18} admits a unique solution $G_i \in \chi^{m_i}(\mathbb{R})$ and  $\|G_{i}\|_{\chi^{m_i}(\mathbb{R})}\leq C\delta$ for any given integer $m_i>0$.
Thus, the proof of Proposition \ref{newlemma2.1} is completed.

\

\noindent
{\bf Acknowledgements\ }
S. Geng's research is supported in part by the National Natural Science Foundation of China (No. 12071397).
 F. Huang's research is supported in part by the
National Natural Science Foundation of China (No. 11688101).

     \end{document}